\documentclass[12pt]{amsart}  
\usepackage{latexsym}
\usepackage{mathrsfs}
\usepackage{amsmath}
\usepackage{graphicx}
\usepackage{amssymb}
\usepackage{color}
\usepackage{amsthm}
\usepackage{txfonts}

\newtheorem{theorem}{Theorem}[section]
\newtheorem{prop}[theorem]{Proposition}
\theoremstyle{definition}
\newtheorem{defn}[theorem]{Definition}
\newtheorem{lemma}[theorem]{Lemma}
\newtheorem{coro}[theorem]{Corollary}
\newtheorem{prop-def}{Proposition-Definition}[section]
\newtheorem{coro-def}{Corollary-Definition}[section]

\newtheorem{remark}[theorem]{Remark}

\newtheorem{exam}[theorem]{Example}

\newcommand{\nc}{\newcommand}
\nc{\tred}[1]{\textcolor{red}{#1}}
\nc{\tblue}[1]{\textcolor{blue}{#1}}
\nc{\tgreen}[1]{\textcolor{green}{#1}}
\nc{\tpurple}[1]{\textcolor{purple}{#1}}
\nc{\btred}[1]{\textcolor{red}{\bf #1}}
\nc{\btblue}[1]{\textcolor{blue}{\bf #1}}
\nc{\btgreen}[1]{\textcolor{green}{\bf #1}}
\nc{\btpurple}[1]{\textcolor{purple}{\bf #1}}
\nc{\NN}{{\mathbb N}}
\nc{\ncsha}{{\mbox{\cyr X}^{\mathrm NC}}} \nc{\ncshao}{{\mbox{\cyrX}^{\mathrm NC}_0}}

\newcommand{\efootnote}[1]{}

\renewcommand{\textbf}[1]{}

\newcommand{\delete}[1]{}

\nc{\mlabel}[1]{\label{#1}}  
\nc{\mcite}[1]{\cite{#1}}  
\nc{\mref}[1]{\ref{#1}}  
\nc{\mbibitem}[1]{\bibitem{#1}} 

\delete{
\nc{\mlabel}[1]{\label{#1}{\hfill \hspace{1cm}{\bf{{\ }\hfill(#1)}}}}
\nc{\mcite}[1]{\cite{#1}{{\bf{{\ }(#1)}}}}  
\nc{\mref}[1]{\ref{#1}{{\bf{{\ }(#1)}}}}  
\nc{\mbibitem}[1]{\bibitem[\bf #1]{#1}} 
}

\newcommand{\tun}{\begin{picture}(5,0)(-2,-1)
\put(0,0){\circle*{2}}
\end{picture}}

\newcommand{\tdeux}{\begin{picture}(7,7)(0,-1)
\put(3,0){\circle*{2}}
\put(3,0){\line(0,1){5}}
\put(3,5){\circle*{2}}
\end{picture}}

\newcommand{\ttroisun}{\begin{picture}(15,8)(-5,-1)
\put(3,0){\circle*{2}}
\put(-0.65,0){$\vee$}
\put(6,7){\circle*{2}}
\put(0,7){\circle*{2}}
\end{picture}}

\newcommand{\tquatredeux}{\begin{picture}(15,18)(-5,-1)
\put(3,0){\circle*{2}}
\put(-0.65,0){$\vee$}
\put(6,7){\circle*{2}}
\put(0,7){\circle*{2}}
\put(0,14){\circle*{2}}
\put(0,7){\line(0,1){7}}
\end{picture}}

\newcommand{\tdun}[1]{\begin{picture}(10,5)(-2,-1)
\put(0,0){\circle*{2}}
\put(3,-2){\tiny #1}
\end{picture}}

\newcommand{\tduns}[1]{\begin{picture}(10,5)(-2,-1)
\put(0,0){\circle*{2}}
\put(3,-2){\tiny $\sigmaup$}
\end{picture}}

\newcommand{\tddeuxs}[2]{\begin{picture}(12,5)(0,-1)
\put(3,0){\circle*{2}}
\put(3,0){\line(0,1){5}}
\put(3,5){\circle*{2}}
\put(6,-2){\tiny $\sigmaup$}
\put(6,3){\tiny $\sigmaup$}
\end{picture}}

\newcommand{\tddeux}[2]{\begin{picture}(12,5)(0,-1)
\put(3,0){\circle*{2}}
\put(3,0){\line(0,1){5}}
\put(3,5){\circle*{2}}
\put(6,-2){\tiny $\sigmaup$}
\put(6,3){\tiny #2}
\end{picture}}

\newcommand{\tddeuxx}[2]{\begin{picture}(12,5)(0,-1)
\put(3,0){\circle*{2}}
\put(3,0){\line(0,1){5}}
\put(3,5){\circle*{2}}
\put(6,-2){\tiny #1}
\put(6,3){\tiny #2}
\end{picture}}

\newcommand{\tdtroisuns}[3]{\begin{picture}(20,12)(-5,-1)
\put(3,0){\circle*{2}}
\put(-0.65,0){$\vee$}
\put(6,7){\circle*{2}}
\put(0,7){\circle*{2}}
\put(5,-2){\tiny $\sigmaup$}
\put(9,5){\tiny $\sigmaup$}
\put(-5,5){\tiny $\sigmaup$}
\end{picture}}

\newcommand{\tdtroisun}[3]{\begin{picture}(20,12)(-5,-1)
\put(3,0){\circle*{2}}
\put(-0.65,0){$\vee$}
\put(6,7){\circle*{2}}
\put(0,7){\circle*{2}}
\put(5,-2){\tiny $\sigmaup$}
\put(9,5){\tiny #2}
\put(-5,5){\tiny #3}
\end{picture}}

\newcommand{\tdtroisunx}[3]{\begin{picture}(20,12)(-5,-1)
\put(3,0){\circle*{2}}
\put(-0.65,0){$\vee$}
\put(6,7){\circle*{2}}
\put(0,7){\circle*{2}}
\put(5,-2){\tiny #1}
\put(9,5){\tiny #2}
\put(-5,5){\tiny #3}
\end{picture}}
\newcommand{\tdtroisdeux}[3]{\begin{picture}(12,12)(-2,-1)
\put(0,0){\circle*{2}}
\put(0,0){\line(0,1){5}}
\put(0,5){\circle*{2}}
\put(0,5){\line(0,1){5}}
\put(0,10){\circle*{2}}
\put(3,-2){\tiny $\sigmaup$}
\put(3,3){\tiny #2}
\put(3,9){\tiny #3}
\end{picture}}

\newcommand{\tdquatredeux}[4]{\begin{picture}(20,20)(-5,-1)
\put(3,0){\circle*{2}}
\put(-.65,0){$\vee$}
\put(6,7){\circle*{2}}
\put(0,7){\circle*{2}}
\put(0,14){\circle*{2}}
\put(0,7){\line(0,1){7}}
\put(5,-2){\tiny $\sigmaup$}
\put(9,5){\tiny #2}
\put(-5,5){\tiny $\sigmaup$}
\put(-5,12){\tiny #4}
\end{picture}}

\newcommand{\tdquatredeuxx}[4]{\begin{picture}(20,20)(-5,-1)
\put(3,0){\circle*{2}}
\put(-.65,0){$\vee$}
\put(6,7){\circle*{2}}
\put(0,7){\circle*{2}}
\put(0,14){\circle*{2}}
\put(0,7){\line(0,1){7}}
\put(5,-2){\tiny #1}
\put(9,5){\tiny #2}
\put(-5,5){\tiny #3}
\put(-5,12){\tiny #4}
\end{picture}}

\newcommand{\tdquatretrois}[4]{\begin{picture}(20,20)(-5,-1)
\put(3,0){\circle*{2}}
\put(-.65,0){$\vee$}
\put(6,7){\circle*{2}}
\put(0,7){\circle*{2}}
\put(6,14){\circle*{2}}
\put(6,7){\line(0,1){7}}
\put(5,-2){\tiny $\sigmaup$}
\put(9,5){\tiny #2}
\put(-5,5){\tiny #4}
\put(9,12){\tiny #3}
\end{picture}}

\newcommand{\tdquatrequatre}[4]{\begin{picture}(20,14)(-5,-1)
\put(3,5){\circle*{2}}
\put(-.65,5){$\vee$}
\put(6,12){\circle*{2}}
\put(0,12){\circle*{2}}
\put(3,0){\circle*{2}}
\put(3,0){\line(0,1){5}}
\put(6,-3){\tiny $\sigmaup$}
\put(6,4){\tiny #2}
\put(9,12){\tiny #3}
\put(-5,12){\tiny #4}
\end{picture}}

\newcommand{\tdquatrequatrex}[4]{\begin{picture}(20,14)(-5,-1)
\put(3,5){\circle*{2}}
\put(-.65,5){$\vee$}
\put(6,12){\circle*{2}}
\put(0,12){\circle*{2}}
\put(3,0){\circle*{2}}
\put(3,0){\line(0,1){5}}
\put(6,-3){\tiny #1}
\put(6,4){\tiny #2}
\put(9,12){\tiny #3}
\put(-5,12){\tiny #4}
\end{picture}}

\nc{\opa}{\ast} \nc{\opb}{\odot} \nc{\op}{\bullet} \nc{\pa}{\frakL}
\nc{\arr}{\rightarrow} \nc{\lu}[1]{(#1)} \nc{\mult}{\mrm{mult}}
\nc{\diff}{\mathfrak{Diff}}
\nc{\opc}{\sharp}\nc{\opd}{\natural}
\nc{\ope}{\circ}
\nc{\dpt}{\mathrm{d}}
\nc{\hck}{H_{RT}}
\nc{\vdf}{\calf}
\nc{\ldf}{\calf_\ell}
\nc{\hlf}{H_\ell}
\nc{\onek}{\mathbf{1}_\bfk}
\nc{\diam}{alternating\xspace}
\nc{\Diam}{Alternating\xspace}
\nc{\cdiam}{canonical alternating\xspace}
\nc{\Cdiam}{Canonical alternating\xspace}
\nc{\AW}{\mathcal{A}}

\nc{\ari}{\mathrm{ar}}

\nc{\lef}{\mathrm{lef}}

\nc{\Sh}{\mathrm{ST}}

\nc{\Cr}{\mathrm{Cr}}

\nc{\st}{{Schr\"oder tree}\xspace}
\nc{\sts}{{Schr\"oder trees}\xspace}

\nc{\vertset}{\Omega} 
\nc{\pb}{{\mathrm{pb}}}
\nc{\Lf}{{\mathrm{Lf}}}
\nc{\lft}{{left tree}\xspace}
\nc{\lfts}{{left trees}\xspace}
\nc{\fat}{{fundamental averaging tree}\xspace}
\nc{\fats}{{fundamental averaging trees}\xspace}
\nc{\avt}{\mathrm{Avt}}
\nc{\rass}{{\mathit{RAss}}}
\nc{\aass}{{\mathit{AAss}}}
\nc{\vin}{{\mathrm Vin}}    
\nc{\lin}{{\mathrm Lin}}    
\nc{\inv}{\mathrm{I}n}
\nc{\gensp}{V} 
\nc{\genbas}{\mathcal{V}} 
\nc{\bvp}{V_P}     
\nc{\gop}{{\,\omega\,}}     

\nc{\bin}[2]{ (_{\stackrel{\scs{#1}}{\scs{#2}}})}  
\nc{\binc}[2]{ \left (\!\! \begin{array}{c} \scs{#1}\\
    \scs{#2} \end{array}\!\! \right )}  
\nc{\bincc}[2]{  \left ( {\scs{#1} \atop
    \vspace{-1cm}\scs{#2}} \right )}  
\nc{\bs}{\bar{S}} \nc{\cosum}{\sqsubset} \nc{\la}{\longrightarrow}
\nc{\rar}{\rightarrow} \nc{\dar}{\downarrow} \nc{\dprod}{**}
\nc{\dap}[1]{\downarrow \rlap{$\scriptstyle{#1}$}}
\nc{\md}{\mathrm{dth}} \nc{\uap}[1]{\uparrow
\rlap{$\scriptstyle{#1}$}} \nc{\defeq}{\stackrel{\rm def}{=}}
\nc{\disp}[1]{\displaystyle{#1}} \nc{\dotcup}{\
\displaystyle{\bigcup^\bullet}\ } \nc{\gzeta}{\bar{\zeta}}
\nc{\hcm}{\ \hat{,}\ } \nc{\hts}{\hat{\otimes}}
\nc{\barot}{{\otimes}} \nc{\free}[1]{\bar{#1}}
\nc{\uni}[1]{\tilde{#1}} \nc{\hcirc}{\hat{\circ}} \nc{\lleft}{[}
\nc{\lright}{]} \nc{\lc}{\lfloor} \nc{\rc}{\rfloor}
\nc{\curlyl}{\left \{ \begin{array}{c} {} \\ {} \end{array}
    \right .  \!\!\!\!\!\!\!}
\nc{\curlyr}{ \!\!\!\!\!\!\!
    \left . \begin{array}{c} {} \\ {} \end{array}
    \right \} }
\nc{\longmid}{\left | \begin{array}{c} {} \\ {} \end{array}
    \right . \!\!\!\!\!\!\!}
\nc{\onetree}{\bullet} \nc{\ora}[1]{\stackrel{#1}{\rar}}
\nc{\ola}[1]{\stackrel{#1}{\la}}
\nc{\ot}{\otimes} \nc{\mot}{{{\boxtimes\,}}}
\nc{\otm}{\overline{\boxtimes}} \nc{\sprod}{\bullet}
\nc{\scs}[1]{\scriptstyle{#1}} \nc{\mrm}[1]{{\rm #1}}
\nc{\margin}[1]{\marginpar{\rm #1}}   
\nc{\dirlim}{\displaystyle{\lim_{\longrightarrow}}\,}
\nc{\invlim}{\displaystyle{\lim_{\longleftarrow}}\,}
\nc{\mvp}{\vspace{0.3cm}} \nc{\tk}{^{(k)}} \nc{\tp}{^\prime}
\nc{\ttp}{^{\prime\prime}} \nc{\svp}{\vspace{2cm}}
\nc{\vp}{\vspace{8cm}} \nc{\proofbegin}{\noindent{\bf Proof: }}
\nc{\proofend}{$\blacksquare$ \vspace{0.3cm}}
\nc{\modg}[1]{\!<\!\!{#1}\!\!>}
\nc{\intg}[1]{F_C(#1)} \nc{\lmodg}{\!
<\!\!} \nc{\rmodg}{\!\!>\!}
\nc{\cpi}{\widehat{\Pi}}
\nc{\sha}{{\mbox{\cyr X}}}  
\nc{\shap}{{\mbox{\cyrs X}}} 
\nc{\shpr}{\diamond}    
\nc{\shp}{\ast} \nc{\shplus}{\shpr^+}
\nc{\shprc}{\shpr_c}    
\nc{\msh}{\ast} \nc{\zprod}{m_0} \nc{\oprod}{m_1}
\nc{\vep}{\varepsilon} \nc{\labs}{\mid\!} \nc{\rabs}{\!\mid}
\nc{\sqmon}[1]{\langle #1\rangle}

\nc{\mmbox}[1]{\mbox{\ #1\ }} \nc{\dep}{\mrm{dep}} \nc{\fp}{\mrm{FP}}
\nc{\rchar}{\mrm{char}} \nc{\End}{\mrm{End}} \nc{\Fil}{\mrm{Fil}}
\nc{\Mor}{Mor\xspace} \nc{\gmzvs}{gMZV\xspace}
\nc{\gmzv}{gMZV\xspace} \nc{\mzv}{MZV\xspace}
\nc{\mzvs}{MZVs\xspace} \nc{\Hom}{\mrm{Hom}} \nc{\id}{\mrm{id}}
\nc{\im}{\mrm{im}} \nc{\incl}{\mrm{incl}} \nc{\map}{\mrm{Map}}
\nc{\mchar}{\rm char} \nc{\nz}{\rm NZ} \nc{\supp}{\mathrm Supp}

\nc{\Alg}{\mathbf{Alg}} \nc{\Bax}{\mathbf{Bax}} \nc{\bff}{\mathbf f}
\nc{\bfk}{{\bf k}} \nc{\bfone}{{\bf 1}} \nc{\bfx}{\mathbf x}
\nc{\bfy}{\mathbf y}
\nc{\base}[1]{\bfone^{\otimes ({#1}+1)}} 
\nc{\Cat}{\mathbf{Cat}}

\nc{\detail}{\marginpar{\bf More detail}
    \noindent{\bf Need more detail!}
    \svp}
\nc{\Int}{\mathbf{Int}} \nc{\Mon}{\mathbf{Mon}}
\nc{\rbtm}{{shuffle }} \nc{\rbto}{{Rota-Baxter }}
\nc{\remarks}{\noindent{\bf Remarks: }} \nc{\Rings}{\mathbf{Rings}}
\nc{\Sets}{\mathbf{Sets}} \nc{\wtot}{\widetilde{\odot}}
\nc{\wast}{\widetilde{\ast}} \nc{\bodot}{\bar{\odot}}
\nc{\bast}{\bar{\ast}} \nc{\hodot}[1]{\odot^{#1}}
\nc{\hast}[1]{\ast^{#1}} \nc{\mal}{\mathcal{O}}
\nc{\tet}{\tilde{\ast}} \nc{\teot}{\tilde{\odot}}
\nc{\oex}{\overline{x}} \nc{\oey}{\overline{y}}
\nc{\oez}{\overline{z}} \nc{\oef}{\overline{f}}
\nc{\oea}{\overline{a}} \nc{\oeb}{\overline{b}}
\nc{\weast}[1]{\widetilde{\ast}^{#1}}
\nc{\weodot}[1]{\widetilde{\odot}^{#1}} \nc{\hstar}[1]{\star^{#1}}
\nc{\lae}{\langle} \nc{\rae}{\rangle}
\nc{\lf}{\lfloor}
\nc{\rf}{\rfloor}

\nc{\QQ}{{\mathbb Q}}
\nc{\RR}{{\mathbb R}} \nc{\ZZ}{{\mathbb Z}}

\nc{\cala}{{\mathcal A}} \nc{\calb}{{\mathcal B}}
\nc{\calc}{{\mathcal C}}
\nc{\cald}{{\mathcal D}} \nc{\cale}{{\mathcal E}}
\nc{\calf}{{\mathcal F}} \nc{\calg}{{\mathcal G}}
\nc{\calh}{{\mathcal H}} \nc{\cali}{{\mathcal I}}
\nc{\call}{{\mathcal L}} \nc{\calm}{{\mathcal M}}
\nc{\caln}{{\mathcal N}} \nc{\calo}{{\mathcal O}}
\nc{\calp}{{\mathcal P}} \nc{\calr}{{\mathcal R}}
\nc{\cals}{{\mathcal S}} \nc{\calt}{{\mathcal T}}
\nc{\calu}{{\mathcal U}} \nc{\calw}{{\mathcal W}} \nc{\calk}{{\mathcal K}}
\nc{\calx}{{\mathcal X}} \nc{\CA}{\mathcal{A}}

\nc{\fraka}{{\mathfrak a}} \nc{\frakA}{{\mathfrak A}}
\nc{\frakb}{{\mathfrak b}} \nc{\frakB}{{\mathfrak B}}
\nc{\frakD}{{\mathfrak D}} \nc{\frakF}{\mathfrak{F}}
\nc{\frakf}{{\mathfrak f}} \nc{\frakg}{{\mathfrak g}}
\nc{\frakH}{{\mathfrak H}} \nc{\frakL}{{\mathfrak L}}
\nc{\frakM}{{\mathfrak M}} \nc{\bfrakM}{\overline{\frakM}}
\nc{\frakm}{{\mathfrak m}} \nc{\frakP}{{\mathfrak P}}
\nc{\frakN}{{\mathfrak N}} \nc{\frakp}{{\mathfrak p}}
\nc{\frakS}{{\mathfrak S}} \nc{\frakT}{\mathfrak{T}}
\nc{\frakX}{{\mathfrak X}}
\nc{\BS}{\mathbb{S
}}

\font\cyr=wncyr10 \font\cyrs=wncyr7
\nc{\xing}[1]{\textcolor{red}{Xing:#1}}
\nc{\meng}[1]{\textcolor{blue}{xiaomeng: #1}}
\nc{\revise}[1]{\textcolor{red}{#1}}

\nc{\ID}{{\rm I}}\nc{\lbar}[1]{\overline{#1}}\nc{\bre}{{\rm bre}}
\nc{\sd}{\cals}\nc{\rb}{\rm RB}\nc{\A}{\rm A}\nc{\LL}{\rm L}\nc{\tx}{\tilde{X}}
\nc{\col}{\Delta_{\epsilon}}\nc{\mul}{m_{RT}}\nc{\ul}{u_{RT}}\nc{\epl}{\varepsilon_{RT}}
\nc{\hl}{H_{RT}}\nc{\arro}[1]{#1}\nc{\px}{P_{\tx}}\nc{\pw}{P_{\mathfrak{w}}}\nc{\pl}{B^+}
\nc{\pp}{\pl}\nc{\ppp}[1]{B^+(#1)}\nc{\dw}{\diamond_{\mathfrak{w}}}\nc{\dl}{\diamond_{\rm \ell}}
\nc{\ncshaw}{\sha^{{\rm NC}}_{\mathfrak{w}}}\nc{\ncshal}{\sha^{{\rm NC}}_{{\rm \ell}}}
\nc{\ver}{\rm V}\nc{\ld}{l}\nc{\del}{\Delta_{{\rm \ell}}}\nc{\epsl}{\varepsilon_{{\rm \ell}}}
\nc{\uul}{u_{{\rm \ell}}}\nc{\oneh}{\mathbf{1}}\nc{\onew}{\mathbf{1}}
\nc{\etree}{1} \nc{\conc}{m_{RT}} \nc{\medmid}{{\,~{\tiny \longmid}~\,}}
\nc{\leql}{\leq_{\text{l}}} \nc{\leqh}{\leq_{\text{h}}}
\nc{\leqhl}{\leq_{\text{h,l}}}  \nc{\lhl}{<_{\text{h,l}}}

\begin{document}

\title[Infinitesimal unitary Hopf algebras and planar rooted forests]{Infinitesimal unitary Hopf algebras and planar rooted forests}
%
\author{Xing Gao}
\address{School of Mathematics and Statistics, Key Laboratory of Applied Mathematics and Complex Systems, Lanzhou University, Lanzhou, Gansu 730000, P.\,R. China}
         \email{gaoxing@lzu.edu.cn}

\author{Xiaomeng Wang}
\address{School of Mathematics and Statistics, Lanzhou University, Lanzhou, Gansu 730000, P.\,R. China}
         \email{wangxm2015@lzu.edu.cn}

\date{\today}
\begin{abstract}
Infinitesimal bialgebras were introduced by Joni and Rota.
An infinitesimal bialgebra is at the same time an algebra and coalgebra, in such a way
that the comultiplication is a derivation. Twenty years after Joni and Rota, Aguiar introduced
the concept of an infinitesimal (non-unitary) Hopf algebra.
In this paper we study infinitesimal unitary bialgebras and infinitesimal unitary Hopf algebras, in contrary to Aguiar's approach.
Using an infinitesimal version of the Hochschild 1-cocycle condition, we prove respectively that a class of decorated planar rooted forests
is the free cocycle infinitesimal unitary bialgebra and free cocycle infinitesimal unitary Hopf algebra
on a set. As an application, we obtain that the planar rooted forests is the free cocycle infinitesimal unitary Hopf algebra on the empty set.
\end{abstract}

\subjclass[2010]{
16W99, 
16S10, 
16T05, 
08B20, 
16T10, 
16T30. 
}

\keywords{Decorated planar rooted trees, infinitesimal bialgebra, infinitesimal Hopf algebra, cocycle condition}

\maketitle

\tableofcontents

\setcounter{section}{0}

\allowdisplaybreaks

\section{Introduction}
The renormalization method is to deal with divergence in physics and mathematics~\mcite{CK1, FGK, GPZ, GPZ1, GZ,GZ1},
such as in Feynman integrals and multiple zeta values. The Connes-Kreimer Hopf algebra of rooted forests
was introduced in~\mcite{CK,Kr} as a baby model of the Hopf algebra of Feynman graphs to study the renormalization of perturbative
quantum field theory. Whereafter, this Hopf algebra was studied extensively as one of the main examples of Hopf algebras used for physics applications~\mcite{Fo0, Gal, Kr2, Mo, ZGG}. It is also related to the Loday-Ronco Hopf algebra~\cite{LR} and Grossman-Larson Hopf algebra~\cite{GL} of rooted trees.
Based on planar rooted forests, a non-commutative version $H_{P,\,R}$ of Connes-Kreimer Hopf algebra was introduced simultaneously in~\mcite{Fo1} and~\mcite{Hol}, satisfying a universal property~\mcite{Fo3}. This universal property was generalized in~\mcite{ZGG} to obtain more general free objects in terms of a class of decorated planar rooted forests.

The concept of algebras with (one or more) linear operators was introduced by A. G. Kurosh~\cite{Ku} by the name of $\Omega$-algebras.
In particular, an algebra with one linear operator is called an operated algebra in~\cite{Gop}, in which
the free operated algebra was also constructed. See also~\cite{BCQ, Gub, GSZ}.
In~\mcite{ZGG}, the authors treated the Hopf algebra of planar rooted forests under the viewpoint of operated algebra, where
the linear operator is the grafting operator $B^+$.
In this view, the universal property of the planar rooted forests characterized in~\cite[Theorem~3]{Fo3} can be rephrased as
the free operated algebra on the empty set.
Further, combined with a specific coproduct, a class of decorated planar rooted forests $H_\ell(\tx)$ gives the free objects in the category of Hopf algebras with a given Hochschild 1-cocycle~\mcite{ZGG}, named cocycle Hopf algebras---operated Hopf algebras satisfying a 1-cocycle condition.

Our aim in the present paper is to introduce an infinitesimal version of the Hopf algebra $H_\ell(\tx)$ constructed in~\mcite{ZGG}.
The concept of an infinitesimal bialgebra originated from Joni and Rota~\mcite{JR} in order to provide an algebraic framework for the calculus of divided differences. More precisely, an infinitesimal bialgebra is a triple $(A,m,\Delta)$ where $(A,m)$ is an associative algebra, $(A,\Delta)$ is a coassociative coalgebra
and for each $a,b\in A$,
\vskip-0.25in
\begin{equation}
\Delta(ab)=a\cdot \Delta(b) + \Delta(a) \cdot b =\sum_{(b)} ab_{(1)}\ot b_{(2)}+\sum_{(a)} a_{(1)}\ot a_{(2)}b.
\mlabel{eq:leibniz}
\end{equation}
\vskip-0.1in
\noindent A typical example is that the path algebra of an arbitrary quiver admits a canonical structure of an infinitesimal bialgebra~\mcite{Ag0}.
Later the basic theories of infinitesimal bialgebras~\mcite{Ag1} and infinitesimal Hopf algebras~\mcite{Ag0,Fo2} were developed,
where analogies with the theories of ordinary Hopf algebras and Lie bialgebras~\mcite{BGS,Dri} were found.

We achieve our aim in two steps. We first equip a new coproduct on a class of decorated planar rooted trees involving an analogy of the Hochschild 1-cocycle condition, which gives a recursive construction of the coproduct in the well-known Connes-Kreimer Hopf algebra of rooted forests~\mcite{Fo3}.
We then extend this new coproduct to a class of decorated planar rooted forests through Eq.~(\mref{eq:leibniz}).
Note that the multiplicative unit is an indispensable ingredient in such 1-cocycle conditions (see Eqs.~(\mref{eq:dbp}) and~(\mref{eq:1cocy}) below), and there is no non-zero infinitesimal bialgebra which is both unitary and counitary~\mcite{Ag0}.
So it is desirable to incorporate the unitary property into infinitesimal bialgebras,
and we propose the concepts of an infinitesimal unitary bialgebra and an infinitesimal unitary Hopf algebra here.
When a 1-cocycle condition is involved, we also propose the concepts of a cocycle infinitesimal unitary bialgebra and a cocycle infinitesimal unitary Hopf algebra.
It is also profitable to incorporate the unitary property into infinitesimal bialgebras to construct free objects.
Namely, we can construct respectively the free objects in the categories of cocycle infinitesimal unitary bialgebras and cocycle infinitesimal unitary Hopf algebras via a class of decorated planar rooted forests, whereas the constructions of free infinisemal bialgebras and free infinitesimal Hopf algebras are still not obtained up to now.

Here is the structure of the paper.
In Section~\mref{sec:CKHOPHAL}, after reviewing basics in planar rooted forests,
we give an infinitesimal version of the 1-cocycle condition (Eq.~(\mref{eq:dbp})), which has
a subtle difference with the usual 1-cocycle condition (Remark~\mref{re:coc}). Thanks to
this new 1-cocycle condition, we obtain a new coproduct $\col$ on a class of decorated planar rooted forests $\hlf(\tx)$,
which, together with the concatenation multiplication, turns the $\hlf(\tx)$ into an infinitesimal unitary bialgebra (Theorem~\mref{thm:rt2}).
To make the coproduct $\col$ more explicit, a combinatorial description of it is also given in Subsection~\mref{subs:combdes}.
Further, we propose the concept of a cocycle infinitesimal unitary bialgebra (Definition~\mref{de:defciub})
and prove that $\hlf(\tx)$ is the free cocycle infinitesimal unitary bialgebra on a set $X$ (Theorem~\mref{thm:propm}).
In Section~\mref{sec:hopf}, continuing the line in Section~\mref{sec:CKHOPHAL},
we start with the concepts of an infinitesimal unitary Hopf algebra (Definition~\mref{de:deha}) and a cocycle infinitesimal unitary Hopf algebra (Definition~\mref{defn:ciuha}). We show that $\hlf(\tx)$ is an infinitesimal unitary Hopf algebra (Theorem~\mref{thm:rt13}) and then the
free cocycle infinitesimal unitary Hopf algebra on a set $X$ (Theorem~\mref{thm:rt16}).
In particular, we obtain that the (undecorated) planar rooted forests is the free cocycle infinitesimal unitary Hopf algebra on the empty set (Corollary~\mref{coro:wdeh}).

{\bf Notation.} In this paper, we will be working over a unitary commutative base ring $\bfk$.
By an algebra we mean an associative algebra (possibly without unit)
and by an coalgebra we mean a coassociative coalgebra (possibly without counit), unless otherwise stated.
Linear maps and tensor products are taken over $\bfk$.
For an algebra $A$, we view $A\ot A$ as an $A$-bimodule via
\begin{equation}
a\cdot(b\otimes c):=ab\otimes c\,\text{ and }\, (b\otimes c)\cdot a:= b\otimes ca.
\mlabel{eq:dota}
\end{equation}

\section{Free cocycle infinitesimal unitary bialgebras of decorated planar rooted forests}
\label{sec:CKHOPHAL}
In this section, we first recall some basic notations used throughout the paper.
Then we show that a class of decorated planar rooted forests
is an infinitesimal unitary bialgebra and further a free cocycle infinitesimal unitary bialgebra.

\subsection{Decorated planar rooted forests}
We expose some concepts and notations on planar rooted forests from~\mcite{Gub,ZGG}.
Let $\calt$ denote the set of planar rooted trees and $S(\calt)$ the free semigroup generated by $\calt$ in which the multiplication is concatenation, denoted by $\mul$ and usually suppressed. Thus an element $F$ in $S(\calt)$, called a {\bf planar rooted forest}, is a noncommutative product of planar rooted trees in $\calt$. Adding to $S(\calt)$ the empty planar rooted tree $1$, we obtain the free monoid $\calf=M(\calt)$.

Next, we review some concepts and notations on decorated planar rooted forest.
Let $X$ be a set and let $\sigmaup$ a symbol not in the set $X$. Denote $\tx:=X\sqcup\{\sigmaup\}$.
For the set $\tx$, let $\calt(\tx)$ (resp. $\calf(\tx):=M(\calt(\tx)))$ denote the set of planar rooted trees (resp. forests) whose vertices,
consisted of leaves and internal vertices, are decorated by elements of $\tx$.

Let $\calt_\ell(\tx)$ (resp. $\ldf(\tx)$)
denote the subset of $\calt(\tx)$ (resp. $\vdf(\tx)$) consisting of (vertex) decorated planar rooted trees (resp. forests) where elements of $X$ decorate the leaves only. In other words, all internal vertices, as well as possibly some of the leaves, are decorated by $\sigmaup$.
Note that the empty tree 1 is in $\calt_\ell(\tx)$. If a tree has only one vertex, then the vertex is treated as a leaf.
Here are some examples in $\calt_\ell(\tx)$ where the root is on the bottom:
$$\tdun{$x$}\, ,\,\tduns{\tiny{\sigmaup}}\, ,\,\tddeux{\tiny{\sigmaup}}{$x$}\, ,\, \tdtroisun{\tiny{\sigmaup}}{$x$}{$\tiny{\sigmaup}$}
\, ,\, \tdquatredeux{\tiny{\sigmaup}}{$\tiny{\sigmaup}$}{\tiny{\sigmaup}}{$x$}\, ,\, \tdquatrequatre{$\tiny{\sigmaup}$}{$\tiny{\sigmaup}$}{$y$}{$x$},$$
whereas the following are some examples not in $\calt_\ell(\tx)$:
$$\tddeuxx{$x$}{$x$}\, ,\, \tdtroisunx{$x$}{$x$}{$\tiny{\sigmaup}$}
\, ,\, \tdquatredeuxx{$\tiny{\sigmaup}$}{$\tiny{\sigmaup}$}{$x$}{$x$}\, ,\, \tdquatrequatrex{$\tiny{\sigmaup}$}{$x$}{$y$}{$x$}.$$
Let $\hlf(\tx):=\bfk \ldf(\tx)=\bfk M(\calt_\ell(\tx))$
be the free $\bfk$-module spanned by $\ldf(\tx)$.
Denote by
$$B^{+}: \hlf(\tx)\rightarrow \hlf(\tx)$$ the grafting map sending $1$ to $\bullet_{\sigmaup}$ and sending a rooted forest in $\hlf(\tx)$
to its grafting with the new root decorated by $\sigmaup$,
and by $\conc$ the concatenation on $\hlf(\tx)$.
Then $\hlf(\tx)$ is closed under the concatenation $\mul$~\cite{ZGG}.
Here are some examples about $B^{+}$ on $\hlf(\tx) $:
\vskip -0.2in
\begin{align*}
B^{+}(1) =\tdun{$\tiny{\sigmaup}$},\, B^{+}(\tdun{$x$})=\tddeux{\tiny{\sigmaup}}{$x$}, \, B^{+}(\tddeux{\tiny{\sigmaup}}{$x$}\tdun{$x$})=\tdquatredeux{\tiny{\sigmaup}}{$x$}{\tiny{\sigmaup}}{$x$}.
\end{align*}

For $F=T_1\cdots T_{n}\in \ldf(\tx)$ with $n\geq 0$ and $T_1,\cdots,T_{n}\in \calt_\ell(\tx)$, we define $\bre(F):=n$
to be the {\bf breadth} of $F$. Here we use the convention that $\bre(1) = 0$.
To define the depth of a decorated planar rooted forests
$F\in \ldf(\tx)$, we give a recursive structure on $\calf_\ell(\tx)$.
Denote $\bullet_{X}:=\{\bullet_{x}\mid x\in X\}$ and set
\begin{equation*}
M_{0}:=M(\bullet_{X}) = S(\bullet_{X}) \sqcup\{1\}.
\end{equation*}
Here $M(\bullet_{X})$ (resp. $S(\bullet_{X})$) denotes the submonoid (resp. subsemigroup) of $\calf(\tx)$ generated by $\bullet_{X}$,
which is also isomorphic to the free monoid (resp. semigroup) generated by $\bullet_{X}$, justifying the abuse of notations.
Assume that $M_n, n\geq 0,$ has been defined, and define
\begin{equation}
M_{n+1}:=M(\bullet_{X}\sqcup \ppp{M_n}).
\mlabel{eq:construct} \notag
\end{equation}
Then we have $M_n\subseteq M_{n+1}$ and
$$ \calf_\ell(\tx) = \dirlim M_{n} = \bigcup \limits^{\infty}_{n=0}M_{n}.$$
Now elements $F\in M_n\setminus M_{n-1}$ are said to have {\bf depth} $n$, denoted by $\dep(F) = n$.
Here are some examples:
\begin{align*}
&\dep(\etree)=0, \dep(\tdun{$x$})=0, \dep(\tdun{$\sigmaup$})=\dep(B^+(\etree))=1,
\dep(\tddeux{$\tiny{\sigmaup}$}{$x$})=\dep(B^+(\tdun{$x$}))=1, \\ &\dep(\tddeux{$\sigmaup$}{$\sigmaup$})=\dep(B^+(B^+(1)))=2,\dep(\tdtroisun{$\tiny{\sigmaup$}}{$x$}{$\tiny{\sigmaup}$})=\dep(B^+(B^+(1)\tdun{$x$}))=2.
\end{align*}

\subsection{Infinitesimal unitary bialgebras}
In this subsection, we obtain an infinitesimal unitary bialgebraic structure on
a class of decorated planar rooted forests.

\subsubsection{A new coalgebra structure on decorated planar rooted forests}
\mlabel{subs:copro}
We define a coproduct $\col$ on $\hlf(\tx)$ recursively on depth.
By linearity, we only need to define $\col(F)$ for $F\in \ldf(\tx)$.
For the initial step of $\dep(F) = 0$, we define
\begin{equation}
\col(F) :=
\left\{
\begin{array}{ll}
0, & \text{ if } F = \etree, \\
\etree \ot \etree, & \text{ if } F = \bullet_x \text{ for some } x\in X,\\
\bullet_{x_{1}}\cdot \col(\bullet_{x_{2}}\cdots\bullet_{x_{m}})+ \col(\bullet_{x_{1}}) \cdot (\bullet_{x_{2}}\cdots\bullet_{x_{m}})
& \text{ if }  F=\bullet_{x_{1}}\cdots \bullet_{x_{m}} \text{ with } m\geq 2 \text{ and }x_i\in X.
\end{array}
\right .
 \mlabel{eq:dele}
\end{equation}
Here in the third case, the definition of $\col$ reduces to the induction on breadth and the dot action is defined in Eq.~(\mref{eq:dota}).

For the induction step of $\dep(F)\geq 1$, if $\bre(F) = 1$, then we may write $F=B^{+}(\lbar{F})$ for some $\lbar{F}\in \ldf(\tx)$ and define
\begin{equation}
\begin{aligned}
\col(F) :=\col B^{+}(\lbar{F}) := \lbar{F} \otimes \etree + (\id\otimes B^{+})\col(\lbar{F}),
\mlabel{eq:dbp}
\end{aligned}
\end{equation}
that is, $ \col B^{+} = \id \ot 1 + (\id\otimes B^{+})\col.$
Here the coproduct $\col(\lbar{F})$ is defined by the induction hypothesis on depth, and we call
Eq.~(\mref{eq:dbp}) the {\bf infinitesimal 1-cocycle condition} (abbreviated {\bf $\epsilon$-cocycle condition}).
If  $\bre(F) \geq 2$, then we may assume $F=T_{1}T_{2}\cdots T_{m}$ with $\bre(F) = m\geq 2$ and define
\begin{equation}
\col(F)=T_{1}\cdot \col(T_{2}\cdots T_{m})+\col(T_{1})\cdot (T_{2}\cdots T_{m}).
\mlabel{eq:dele1}
\end{equation}
Here the definition of $\col$ again reduces to the induction on breadth.

\begin{exam}
Foissy~\mcite{Fo2} also studied another kind of infinitesimal Hopf algebras on (undecorated) planar rooted forests, using a different coproduct $\Delta_\calf$ given by
\begin{equation*}
\Delta_{\calf}(F) :=
\left\{
\begin{array}{ll}
1\ot 1, & \text{ if } F = \etree, \\
F\ot 1+(\id\ot B^{+})\Delta_\calf(\lbar{F}), & \text{ if }  F= B^+(\lbar{F}),\\
F_{1} \cdot \Delta_\calf(F_{2})+ \Delta_\calf (F_{1}) \cdot F_2 -F_{1}\ot F_{2}, & \text{ if } F = F_{1}F_{2}.
\end{array}
\right .
\end{equation*}
We give some examples to expose the differences between these two coproducts $\col$ and $\Delta_\calf$.
On the one hand,
\begin{align*}
 \col\left(\tddeux{$\tiny{\sigmaup}$}{$x$}\right)&=\tdun{$x$}\ot 1+1\ot \tduns{$\tiny {\sigmaup}$};\\
 \col\left(\tdtroisun{$\tiny{\sigmaup}$}{$x$}{$\tiny{\sigmaup}$}\right)&=\tduns{$\tiny {\sigmaup}$}\tdun{$x$}\ot 1+\tduns{$\tiny {\sigmaup}$}\ot\tduns{$\tiny {\sigmaup}$}+1\ot \tddeux{$\tiny{\sigmaup}$}{$x$};\\
 \col\left(\tdquatredeux{$\tiny{\sigmaup}$}{$\tiny{\sigmaup}$}{$\tiny{\sigmaup}$}{$x$}\right)&=\tddeux{$\tiny{\sigmaup}$}{$x$}\tduns{$\tiny {\sigmaup}$}\ot 1+\tddeux{$\tiny{\sigmaup}$}{$x$}\ot\tduns{$\tiny {\sigmaup}$}+\tdun{$x$}\ot\tddeuxs{$\tiny{\sigmaup}$}{$\tiny{\sigmaup}$}+1\ot\tdtroisuns{$\tiny{\sigmaup}$}{$\tiny{\sigmaup}$}{$\tiny{\sigmaup}$}.
 \end{align*}
On the other hand,
 \begin{align*}
\Delta_{\calf}\left(\tdeux\right)&=\tdeux\ot 1+1\ot\tdeux+\tun\ot\tun;\\
\Delta_{\calf}\left(\ttroisun\right)&=\ttroisun\ot 1+1\ot\ttroisun+\tun\tun\ot\tun+\tun\ot\tdeux;\\
\Delta_{\calf}\left(\tquatredeux\right)&=\tquatredeux\ot 1+1\ot\tquatredeux+\tdeux\tun\ot\tun+\tun\ot\ttroisun+\tdeux\ot\tdeux.
 \end{align*}
 \mlabel{ex:excol}
\end{exam}

\begin{remark}
The coproduct $\Delta_{RT}$ on $\hlf(\tx)$ given in~\mcite{ZGG} is defined by $\Delta_{RT}(\etree) = \etree\ot \etree$ and the {\bf cocycle condition}
\begin{equation}
\Delta_{RT} B^{+}(\lbar{F}) := F \otimes \etree + (\id\otimes B^{+})\Delta_{RT}(\lbar{F})\,\text{ for } F\in \ldf(\tx).
\mlabel{eq:1cocy}
\end{equation}
Note the subtle difference between this cocycle condition and the $\epsilon$-cocycle condition in Eq.~(\mref{eq:dbp}).
\mlabel{re:coc}
\end{remark}

We are going to show that $\hlf(\tx)$ equipped with the $\col$ is a coalgebra.
The following two lemmas are needed.

\begin{lemma}
Let $\bullet_{x_{1}}\cdots \bullet_{x_{m}}\in \hlf(\tx)$ with $m\geq 1$ and $x_1, \cdots, x_m\in X$. Then
$$\col(\bullet_{x_{1}}\cdots \bullet_{x_{m}})=\sum_{i=1}^m\bullet_{x_{1}}\cdots\bullet_{x_{i-1}}\otimes\bullet_{x_{i+1}}\cdots\bullet_{x_{m}},$$
with the convention that $\bullet_{x_1} \bullet_{x_0} = \etree$ and $\bullet_{x_{m+1}} \bullet_{x_m} = \etree$.
\mlabel{lem:rt11}
\end{lemma}

\begin{proof}
We prove the result by induction on $m\geq 1$. For the initial step of $m=1$, we have
$$\col(\bullet_{x_{1}})=1\otimes 1,$$
and the result is true trivially. For the induction step of $m\geq 2$, we get
\allowdisplaybreaks{
\begin{align*}
\col(\bullet_{x_{1}}\cdots \bullet_{x_{m}}) =& \bullet_{x_{1}} \cdot \col(\bullet_{x_{2}}\cdots\bullet_{x_{m}})+ \col(\bullet_{x_{1}})\cdot (\bullet_{x_{2}}\cdots\bullet_{x_{m}})  \quad (\text{by Eq.~(\mref{eq:dele})})\\
=& \bullet_{x_{1}} \cdot \col(\bullet_{x_{2}}\cdots\bullet_{x_{m}})+ (\etree\ot \etree)\cdot (\bullet_{x_{2}}\cdots\bullet_{x_{m}})  \quad (\text{by Eq.~(\mref{eq:dele})})\\
=&\bullet_{x_{1}} \cdot \col(\bullet_{x_{2}}\cdots\bullet_{x_{m}}) + \etree \otimes\bullet_{x_{2}}\cdots\bullet_{x_{m}} \quad (\text{by Eq.~(\mref{eq:dota})})\\
=&\bullet_{x_{1}} \cdot \left( \sum_{i = 2}^m \bullet_{x_{2}}\cdots\bullet_{x_{i-1}}\otimes\bullet_{x_{i+1}}\cdots\bullet_{x_{m}}\right) + \etree \otimes\bullet_{x_{2}}\cdots\bullet_{x_{m}} \quad (\text{by the induction hypothesis})\\
=&\sum_{i =2}^m \bullet_{x_{1}}\cdots\bullet_{x_{i-1}}\otimes\bullet_{x_{i+1}}\cdots\bullet_{x_{m}}+1\otimes\bullet_{x_{2}}\cdots\bullet_{x_{m}}  \quad (\text{by Eq.~(\mref{eq:dota})})\\
=&\sum_{i=1}^m\bullet_{x_{1}}\cdots\bullet_{x_{i-1}}\otimes\bullet_{x_{i+1}}\cdots\bullet_{x_{m}},
\end{align*}
}
as required.
\end{proof}

\begin{lemma}
Let $F_1, F_2\in \ldf(\tx)$. Then $\col(F_1 F_2) = F_1 \cdot \col(F_2) + \col(F_1) \cdot F_2$.
\mlabel{lem:colff}
\end{lemma}

\begin{proof}
Consider first that $\bre(F_{1})=0$ or $\bre(F_{2}) = 0$.
Without loss of generality, let $\bre(F_{2}) = 0$. Then $F_2=1$ and
\begin{equation*}
\col(F_{1}F_{2})=\col(F_{1}\etree)=\col(F_{1})=\col(F_{1})\cdot\etree=F_{1}\cdot\col(\etree)+\col(F_{1})\cdot\etree,
\end{equation*}
where the last step employs $\col(\etree) = 0$ in Eq.~(\mref{eq:dele}).

Consider next that $\bre(F_{1})=m_{1}\geq 1$ and $\bre(F_{2})=m_{2}\geq 1$.
In this case, we prove the result by induction on $m_{1}+m_{2}\geq 2$. For the initial step of $m_{1}+m_{2}=2$, we have $m_{1}=1$ and $m_{2}=1$. Then $F_{1}=T_{1}$ and $F_{2}=T_{2}$
for some decorated planar rooted trees $T_1, T_2\in \calt_\ell(\tx)$. So by Eq.~(\mref{eq:dele1}),
\begin{equation*}
\col(F_{1}F_{2})=\col(T_{1}T_{2})=T_{1}\cdot\col(T_{2})+\col(T_{1})\cdot T_{2}=F_{1}\cdot\col(F_{2})+\col(F_{1})\cdot F_{2}.
\end{equation*}
Assume the result is true for $m_{1}+m_{2}<m$ and consider the case of $m_{1}+m_{2}=m\geq 3$.
By symmetry, we can let $m_1\geq 2$ and write
$F_{1}=T_{1}F_{1}^{'}$ with $\bre(T_{1}) = 1$ and $\bre(F_{1}^{'}) = m_1-1$.
Then
\begin{align*}
\col(F_{1}F_{2})=&\ \col(T_{1}F_{1}^{'} F_2)\\
=&\ T_{1}\cdot\col(F_{1}^{'} F_2)+
\col(T_{1})\cdot (F_{1}^{'} F_2)\quad  (\text{by Eq.~(\ref{eq:dele1})})\\
=&\ T_{1}\cdot\Big(F_{1}^{'}\cdot\col(F_{2})+\col(F_{1}^{'})\cdot F_{2}\Big)+\col(T_{1})\cdot (F_{1}^{'}F_{2})\quad(\text{by the induction hypothesis})\\
=&\ (T_{1} F_{1}^{'}) \cdot\col(F_{2})+T_{1}\cdot\col(F_{1}^{'})\cdot F_{2}+ \big(\col(T_{1})\cdot F_{1}^{'}\big) \cdot F_{2}\\
=&\ F_{1}\cdot\col(F_{2})+ \Big(T_{1}\cdot\col(F_{1}^{'})+\col(T_{1})\cdot F_{1}^{'}\Big)\cdot F_{2}\\
=&\ F_{1}\cdot\col(F_{2})+\col(F_{1})\cdot F_{2}\quad(\text{by the induction hypothesis}).
\end{align*}
This completes the proof.
\end{proof}

Next, we prove that $\hlf(\tx)$ is closed under the coproduct $\col$.

\begin{lemma}
For each $F\in \hlf(\tx)$, we have $\col(F)\in \hlf(\tx)\ot\hlf(\tx)$.
\mlabel{lem:dethh}
\end{lemma}

\begin{proof}
By linearity, it suffices to consider basis elements $F\in \ldf(\tx)$.
We prove the result by induction on $\dep(F)\geq 0$ for $F\in \ldf(\tx)$.
For the initial step of $\dep(F)=0$, we have $F=\bullet_{x_1}\cdots\bullet_{x_{m}}$
for some $m\geq 0$ and $x_1, \cdots, x_m\in X$. Here we use the convention that $F=\etree$
when $m=0$. Then by Lemma~\mref{lem:rt11},
$$\col(F)=\col(\bullet_{x_1}\cdots\bullet_{x_{m}})=\sum_{i=1}^m\bullet_{x_{1}}\cdots \bullet_{x_{i-1}}\ot \bullet_{x_{i+1}}\cdots\bullet_{x_m}\in\hlf(\tx)\ot\hlf(\tx).$$

Assume the result is true for $\dep(F) < m$ and consider the case of $\dep(F) =m\geq 1$.
For this case, we apply the induction on breadth $\bre(F)$. Since $\dep(F) = m\geq1$, we have $F\neq 1$ and $\bre(F)\geq 1$.
If $\bre(F)=1$, we have $F = B^{+}(\lbar{F})$ for some $\lbar{F} \in \hlf(\tx)$ and
\begin{align*}
\col(F)= \col\Big(B^{+}(\lbar{F})\Big) = \lbar{F}\ot 1+(\id\ot B^{+})\col(\lbar{F}).
\end{align*}
Note that $\lbar{F}\ot 1\in \hlf(\tx) \ot \hlf(\tx)$ by $\lbar{F}\in \hlf(\tx)$. Moreover by the induction hypothesis,
$$\col(\lbar{F})\in \hlf(\tx)\ot \hlf(\tx) \,\text{ and so }\, (\id\ot B^{+})\col(\lbar{F})\in \hlf(\tx)\ot\hlf(\tx).$$
Hence
$$\lbar{F}\ot 1+(\id\ot B^{+})\col(\lbar{F})\in \hlf(\tx)\ot\hlf(\tx).$$
Assume the result is true for $\dep(F) = m$ and $\bre(F)<k$, and consider the case of
$\dep(F) = m$ and $\bre(F) = k\geq 2$.
Then we may write $F=T_{1}T_{2}\cdots T_{k}$ with $T_1, \cdots, T_k\in \calt_\ell(\tx)$ and so
\begin{align*}
\col(F)=&\col(T_{1}T_{2}\cdots T_{k})\\
=&\ T_{1}\cdot \col(T_{2}\cdots T_{k})+\col(T_{1})\cdot (T_{2}\cdots T_{k})\\
=& \ (T_{1}T_2) \cdot \col(T_{3}\cdots T_{k})+ T_{1} \cdot \col(T_{2}) \cdot (T_3\cdots T_{k})  + \col(T_{1})\cdot (T_{2}\cdots T_{k})\\
=& \cdots = \ \sum_{i =1}^k (T_{1} \cdots T_{i-1}) \cdot\col(T_{i}) \cdot (T_{i+1}\cdots T_{k}).
\end{align*}
For each $i=1, \cdots, k$, by the induction on breadth, we have
$$ \col(T_{i}) \in \hlf(\tx)\ot \hlf(\tx),$$
and whence by Eq.~(\mref{eq:dota}),
$$(T_{1} \cdots T_{i-1}) \cdot\col(T_{i}) \cdot (T_{i+1}\cdots T_{k})\in \hlf(\tx)\ot \hlf(\tx).$$
This completes the proof.
\end{proof}

\begin{theorem}
The pair $(\hlf(\tx), \col)$ is a coalgebra (without counit).
\mlabel{thm:rt1}
\end{theorem}

\begin{proof}
It suffices to prove the coassociative law
\begin{equation}
(\id\otimes \col)\col(F)=(\col\otimes \id)\col(F)\,\text{ for } F\in \ldf(\tx)
\mlabel{eq:coass}
\end{equation}
by induction on $\dep(F)\geq 0$. For the initial step of $\dep(F)=0$, we have
$F=\bullet_{x_{1}} \cdots\bullet_{x_{m}}$ for some $m\geq 0$ and $x_1, \cdots, x_m\in X$. Then
\allowdisplaybreaks{
\begin{align*}
&(\id\ot\col)\col(\bullet_{x_{1}} \cdots\bullet_{x_{m}})\\
=&\  (\id\ot\col)\Big(\sum_{i=1}^m\bullet_{x_{1}} \cdots\bullet_{x_{i-1}}\otimes\bullet_{x_{i+1}} \cdots\bullet_{x_{m}}\Big)\quad(\text{by Lemma~\mref{lem:rt11}})\\
=&\ (\id\ot\col)\Big(\sum_{i=1}^{m-1}\bullet_{x_{1}} \cdots\bullet_{x_{i-1}}\otimes\bullet_{x_{i+1}} \cdots\bullet_{x_{m}}+
\bullet_{x_{1}} \cdots\bullet_{x_{m-1}}\ot 1\Big)\\
=&\ \sum_{i=1}^{m-1} \bullet_{x_{1}} \cdots\bullet_{x_{i-1}}\otimes\col(\bullet_{x_{i+1}} \cdots\bullet_{x_{m}})\quad(\text{by $\col(1)=0$ in Eq.~(\ref{eq:dele})})\\
=&\ \sum_{i=1}^{m-1} \bullet_{x_{1}} \cdots\bullet_{x_{i-1}}\otimes\left(\sum_{j=i+1}^m\bullet_{x_{i+1}}\cdots\bullet_{x_{j-1}}\ot\bullet_{x_{j+1}}\cdots\bullet_{x_{m}}\right)\quad(\text{by Lemma~\mref{lem:rt11}})\\
=&\sum_{i=1}^{m-1}\sum_{j= i+1}^{m} \bullet_{x_{1}}\cdots\bullet_{x_{i-1}}\otimes\bullet_{x_{i+1}}\cdots\bullet_{x_{j-1}}\otimes\bullet_{x_{j+1}}\cdots\bullet_{x_{m}}\\
=&\sum_{j=2}^m\sum_{i=1}^{j-1}\bullet_{x_{1}}\cdots\bullet_{x_{i-1}}\otimes\bullet_{x_{i+1}}\cdots\bullet_{x_{j-1}}\otimes
\bullet_{x_{j+1}}\cdots\bullet_{x_{m}}+\col(1)\ot\bullet_{x_{2}}\cdots\bullet_{x_{m}} \quad(\text{by $\col(1)=0$ in Eq.~(\ref{eq:dele})})\\
=&\sum_{j=2}^m\Big(\sum_{i= 1}^{j-1}\bullet_{x_{1}}\cdots\bullet_{x_{i-1}}\otimes\bullet_{x_{i+1}}\cdots\bullet_{x_{j-1}}\Big)\otimes\bullet_{x_{j+1}}\cdots\bullet_{x_{m}}+\col(1)\ot\bullet_{x_{2}}\cdots\bullet_{x_{m}}\\
=&\sum_{j=2}^m\col(\bullet_{x_{1}}\cdots\bullet_{x_{j-1}})\otimes\bullet_{x_{j+1}}\cdots\bullet_{x_{m}}+\col(1)\ot\bullet_{x_{2}}\cdots\bullet_{x_{m}}\quad(\text{by Lemma~\mref{lem:rt11}})\\
=&\
(\col\ot\id)\Big(\sum_{j=2}^m\bullet_{x_{1}}\cdots\bullet_{x_{j-1}}\ot\bullet_{x_{j+1}}\cdots\bullet_{x_{m}}+1\ot\bullet_{x_{2}}\cdots\bullet_{x_{m}}\Big)\\
=&\ (\col\otimes\id)\Big(\sum_{j=1}^{m}\bullet_{x_{1}}\cdots\bullet_{x_{j-1}}\otimes\bullet_{x_{j+1}}\cdots\bullet_{x_{m}}\Big)\\
=&\ (\col\otimes\id)\col(\bullet_{x_{1}}\cdots\bullet_{x_{m}})\quad(\text{by Lemma~\mref{lem:rt11}}).
\end{align*}
}
Assume that Eq.~(\mref{eq:coass}) is valid for $\dep(F)\leq n$ and consider the case of $\dep(F)=n+1$.
We now apply the induction on breadth $\bre(F)$. As $\dep(F)\geq 1$, we get $F\neq 1$ and $\bre(F)\geq 1$.
When $\bre(F)=1$, we may write $F=B^{+}(\lbar{F})$ for some $\lbar{F}\in \ldf(\tx)$ and
\allowdisplaybreaks{
\begin{align*}
(\id\otimes\col)\col(F)=&\ (\id\otimes \col)\col(B^{+}(\lbar{F}))\\
=&\ (\id\otimes \col)\Big(\lbar{F}\otimes \etree+(\id\otimes B^{+})\col(\lbar{F})\Big)\quad(\text{by Eq.~(\mref{eq:dbp})})\\
=&\ \lbar{F}\otimes \col(\etree)+(\id\otimes (\col B^{+}))\col(\lbar{F})\\
=&\  (\id\otimes (\col B^{+}))\col(\lbar{F})  \quad(\text{by $\col(1)=0$ in Eq.~(\ref{eq:dele})})\\
=&\   \Big(\id\otimes \big(\id\ot \etree + (\id \ot B^+)\col \big) \Big)\col(\lbar{F}) \quad(\text{by Eq.~(\mref{eq:dbp})})\\
=&\ \Big(\id\otimes \id\otimes \etree+(\id\otimes \id\otimes B^{+})(\id\otimes \col)\Big)\col(\lbar{F})\\
=&\ (\id\otimes \id\otimes \etree)\col(\lbar{F})+(\id\otimes \id\otimes B^{+})(\col\otimes \id)\col(\lbar{F})\quad(\text{by the induction on \dep(F)})\\
=&\ \col(\lbar{F})\otimes \etree+(\col\otimes B^{+})\col(\lbar{F})\\
=&\ (\col\ot\id)\Big(\lbar{F}\ot \etree+(\id\ot B^{+})\col(\lbar{F})\Big) \\
=&\ (\col\otimes \id)\col(F) \quad(\text{by Eq.~(\mref{eq:dbp})}).
\end{align*}
}
Assume that Eq.~(\mref{eq:coass}) holds for $\dep(F)=n+1$ and $\bre(F)\leq m$.
Consider the case when $\dep(F)=n+1$ and $\bre(F)=m+1\geq 2$.
Then we can write $F=F_{1}F_{2}$ for some $F_1, F_2\in \ldf(\tx)$ with $\bre(F_1), \bre(F_2)< \bre(F)$. Hence
\allowdisplaybreaks{
\begin{align*}
&(\id\ot\col)\col(F_1F_2)\\
=&\ (\id\ot\col)(F_1 \cdot \col(F_2)+\col(F_1) \cdot F_2)\\
=&\ (\id\ot\col)\Big(\sum_{(F_2)}F_1F_{2(1)}\ot F_{2(2)}+\sum_{(F_1)}F_{1(1)}\ot F_{1(2)}F_2\Big) \quad(\text{by Eq.~(\mref{eq:dota})})\\
=&\ \sum_{(F_2)}F_1F_{2(1)}\ot\col(F_{2(2)})+\sum_{(F_1)}F_{1(1)}\ot\col(F_{1(2)}F_2)\\
=&\ \sum_{(F_2)}F_1F_{2(1)}\ot\col(F_{2(2)})+\sum_{(F_{1})}F_{1(1)}\ot \Big( F_{1(2)}\cdot \col (F_{2})+\col(F_{1(2)})\cdot F_{2} \Big)\\
=&\ \sum_{(F_2)}F_1F_{2(1)}\ot F_{2(2)}\ot F_{2(3)}+\sum_{(F_1)}F_{1(1)}\ot\sum_{(F_2)}F_{1(2)}F_{2(1)}\ot F_{2(2)}\\
&+\sum_{(F_1)}F_{1(1)}\ot F_{1(2)}\ot F_{1(3)}F_2 \quad(\text{by the induction on breadth and Eq.~(\ref{eq:dota})})\\
=&\ \sum_{(F_2)}F_1F_{2(1)}\ot F_{2(2)}\ot F_{2(3)}+\sum_{(F_2)}\sum_{(F_1)}F_{1(1)}\ot F_{1(2)}F_{2(1)}\ot F_{2(2)}\\
&+\sum_{(F_1)}F_{1(1)}\ot F_{1(2)}\ot F_{1(3)}F_2\\
=&\ \sum_{(F_2)}\col(F_1F_{2(1)})\ot F_{2(2)}+\sum_{(F_1)}\col(F_{1(1)})\ot F_{1(2)}F_2 \quad(\text{by the induction on breadth})\\
=&\ (\col\ot\id)(\sum_{(F_2)}F_1F_{2(1)}\ot F_{2(2)}+\sum_{(F_1)}F_{1(1)}\ot F_{1(2)}F_2)\\
=&\ (\col\ot\id)(F_1 \cdot \col(F_2)+\col(F_1) \cdot F_2)  \quad(\text{by Eq.~(\mref{eq:dota})})\\
=&\ (\col\ot\id)\col(F_1F_2) \quad(\text{by Lemma~\mref{lem:colff}}).
\end{align*}
}
This completes the induction on breadth and the induction on depth.
\end{proof}

\subsubsection{A combinatorial description of $\col$}
\mlabel{subs:combdes}
This subsection is devoted to a combinatorial description of the coproduct $\col$ given recursively in Subsection~\mref{subs:copro},
as in the case of the coproduct in the Connes-Kreimer Hopf algebra by admissible cuts and in the paper of Foissy for his version of infinitesimal Hopf algebra.
To get this explicit construction, first note that the number of terms in the coproduct $\col(F)$ of a decorated planar rooted forest $F$
equals the number of vertices of $F$, which suggests that the coproduct $\col(F)$ depends on its effect on each vertex. To make this more precise,
let us set up some order relations on the vertices of a decorated planar rooted forest, which was introduced by Foissy~\cite{Fo1,Fo2}
in the case of (undecorated) planar rooted forests. For a forest $F$, denote by $V(F)$ the set of vertices of $F$.

\begin{defn}
Let $F = T_1 \cdots T_n\in \ldf(\tx)$ with $T_1, \cdots, T_n\in \calt_\ell(\tx)$ and $n\geq 1$,
and let $a,b\in V(F)$ be two vertices. Then
\begin{enumerate}
\item $a\leqh b$ ({\bf being higher}) if there exists a (directed) path from $a$ to $b$ in $F$ and the edges of $F$ being oriented from roots to leaves;
\item $a\leql b$ ({\bf being more on the left}) if $a$ and $b$ are not comparable for $\leqh $ and one of the following assertions is satisfied:
   \begin{enumerate}
   \item  $a$ is a vertex of $T_i$ and $b$ is a vertex of $T_j$ with $1\leq j< i \leq n$.
   \item $a$ and $b$ are vertices of the same $T_i$, and $a\leql b$ in the forest obtained from $T_i$ by deleting its root;
    \end{enumerate}
\item $a\leqhl b$ ({\bf being higher or more on the left}) if $a\leqh b$ or $a\leql b$.
\end{enumerate}
\mlabel{def:order}
\end{defn}

The $\leqh $ and $\leql$ are partial orders on $V(F)$, and the $\leqhl$ is a linear order on $V(F)$~\cite{Fo1,Fo2}.
As usual, we denote $a\lhl b$ if $a\leqhl b$ but $a\neq b$.

Let $G$ be a graph and $S$ a set of vertices of $G$.
The {\bf induced subgraph} in $G$ by $S$ is the graph whose vertex set is $S$ and whose edge set consists of all of the edges in $G$ that have both endpoints in $S$~\mcite{Dies}. In particular, the induced subgraph by the empty set is the empty graph.
Let $F\in \ldf(\tx)$ be a decorated planar rooted forest.
Let us agree to view $F$ as a graph.
For each vertex $a\in V(F)$, denote by $B_{a}$ the induced subgraph in $F$ by the set $\{b\in V(F) \mid a \lhl b\}$, and by $R_{a}$ the induced subgraph in $F$ by the set $V(F)\setminus (V(B_a)\cup\{ a \})$.
Equivalently, $R_a$ is the induced subgraph in $F$ by the set $\{b\in V(F) \mid b\lhl a\}$.
Note that both $B_a$ and $R_a$ are decorated planar rooted forests in $\ldf(\tx)$, not containing the vertex $a$.
Now we are ready for the combinatorial description of $\col$:
\begin{equation}
\col(F) =\sum_{a\in V(F)} B_{a}\ot R_{a}\,\text{ for }\, F\in \ldf(\tx).
\mlabel{eq:eqcombint}
\end{equation}
Here we use the convention that $\col(F)=0$ when $F=\etree$.
Before we go on to prove that this combinatorial description coincides with recursive definition given in Subsection~\mref{subs:copro},
let us compute some examples for better understanding of Eq.~(\mref{eq:eqcombint}).

\begin{exam}
Consider $F=\tdquatretrois{$\tiny{\sigmaup}$}{$\tiny{\sigmaup}$}{$x$}{$y$}$ with $x,y\in X$.
Then $$\tduns{$\tiny {\sigmaup}$}\ (\text{the root}) \lhl \tduns{$\tiny {\sigmaup}$} \ (\text{not the root}) \lhl \tdun{$x$} \lhl \tdun{$y$}.$$
If $a=\tduns{$\tiny {\sigmaup}$}$ (the root),
then
$$B_{a}=\tdun{$y$}\tddeux{$\tiny{\sigmaup}$}{$x$}\,\text{ and }\, R_a=\etree.$$
If $a=\tdun{$y$}$, then
$$B_{a}=\etree\,\text{ and }\, R_a=\tdtroisdeux{$\tiny{\sigmaup}$}{$\tiny{\sigmaup}$}{$x$}.$$
If $a=\tduns{$\tiny {\sigmaup}$}$ (not the root), then
$$B_{a}=\tdun{$y$}\tdun{$x$}\,\text{ and }\, R_a=\tduns{$\tiny {\sigmaup}$} \ (\text{the root}).$$
If $a=\tdun{$x$}$, then
$$B_{a}=\tdun{$y$}\,\text{ and }\, R_a=\tddeux{$\tiny{\sigmaup}$}{$\tiny{\sigmaup}$}.$$
Consequently,
$$\col(\tdquatretrois{$\tiny{\sigmaup}$}{$\tiny{\sigmaup}$}{$x$}{$y$})=\tdun{$y$}\tddeux{$\tiny{\sigmaup}$}{$x$}\ot \etree+1\ot\tdtroisdeux{$\tiny{\sigmaup}$}{$\tiny{\sigmaup}$}{$x$}+\tdun{$y$}\tdun{$x$}\ot \tduns{$\tiny {\sigmaup}$}+\tdun{$y$}\ot\tddeux{$\tiny{\sigmaup}$}{$\tiny{\sigmaup}$},$$
which is different from the one computed by admissible cuts~\mcite{Fo3,ZGG}:
$$\Delta(\tdquatretrois{$\tiny{\sigmaup}$}{$\tiny{\sigmaup}$}{$x$}{$y$})=
\tdquatretrois{$\tiny{\sigmaup}$}{$\tiny{\sigmaup}$}{$x$}{$y$}\ot 1
+ \tdun{$x$}\ot\tdtroisun{$\tiny{\sigmaup}$}{$\sigmaup$}{$\tiny{y}$}
+\tdun{$y$}\ot\tdtroisdeux{$\tiny{\sigmaup}$}{$\tiny{\sigmaup}$}{$x$}
+\tddeux{$\tiny{\sigmaup}$}{$x$}\ot\tddeux{$\tiny{\sigmaup}$}{$y$}
+\tdun{$y$}\tdun{$x$}\ot\tddeux{$\tiny{\sigmaup}$}{$\tiny{\sigmaup}$}
+ \tdun{$y$}\tddeux{$\tiny{\sigmaup}$}{$x$}\ot \tduns{$\tiny {\sigmaup}$}
+1\ot \tdquatretrois{$\tiny{\sigmaup}$}{$\tiny{\sigmaup}$}{$x$}{$y$}.$$
\end{exam}

\begin{exam}
Let $F=\tdtroisun{$\tiny{\sigmaup}$}{$x$}{$\tiny{\sigmaup}$}\tdquatretrois{$\tiny{\sigmaup}$}{$\tiny{\sigmaup}$}{$z$}{$y$}\tddeux{$\tiny{\sigmaup}$}{$w$}$
with $x,y,z,w\in X$.
Denote by $T_{1}=\tdtroisun{$\tiny{\sigmaup}$}{$x$}{$\tiny{\sigmaup}$}$, $T_{2}=\tdquatretrois{$\tiny{\sigmaup}$}{$\tiny{\sigmaup}$}{$z$}{$y$}$
and $T_{3}=\tddeux{$\tiny{\sigmaup}$}{$w$}$. Then we have the order:
\begin{align*}
&\tduns{$\tiny{\sigmaup}$}\ (\text{the root of $T_{3}$})\lhl\tdun{$w$}\lhl\tdun{$\tiny{\sigmaup}$}\ (\text{the root of $T_{2}$})\lhl\tduns{$\tiny{\sigmaup}$}\ (\text{not the root in $T_{2}$})\lhl\tdun{$z$}\\
&\lhl\tdun{$y$}\lhl\tduns{$\tiny{\sigmaup}$}\ (\text{the root of $T_{1}$})\lhl\tdun{$x$}\lhl\tduns{$\tiny{\sigmaup}$}\ (\text{the leaf in $T_{1}$}).
\end{align*}
If $a=\tduns{$\tiny {\sigmaup}$}$ (the root of $T_{1}$), then
$$B_{a}=\tduns{$\tiny {\sigmaup}$}\tdun{$x$}\,\text{ and }\, R_a=\tdquatretrois{$\tiny{\sigmaup}$}{$\tiny{\sigmaup}$}{$z$}{$y$}\tddeux{$\tiny{\sigmaup}$}{$w$}.$$
If $a=\tduns{$\tiny {\sigmaup}$}$ (the root of $T_{2}$), then
$$B_{a}=\tdtroisun{$\tiny{\sigmaup}$}{$x$}{$\tiny{\sigmaup}$}\tdun{$y$}\tddeux{$\tiny{\sigmaup}$}{$z$}\,\text{ and }\, R_a=\tddeux{$\tiny{\sigmaup}$}{$w$}.$$
If $a=\tduns{$\tiny {\sigmaup}$}$ (not the root in $T_{2}$), then
$$B_{a}=\tdtroisun{$\tiny{\sigmaup}$}{$x$}{$\tiny{\sigmaup}$}\tdun{$y$}\tdun{$z$}\,\text{ and }\, R_a=\tduns{$\tiny {\sigmaup}$}\tddeux{$\tiny{\sigmaup}$}{$w$}.$$
Repeat this process until $a$ runs over $V(F)$ and conclude
\begin{align*}
\col(\tdtroisun{$\tiny{\sigmaup}$}{$x$}{$\tiny{\sigmaup}$}\tdquatretrois{$\tiny{\sigmaup}$}{$\tiny{\sigmaup}$}{$z$}{$y$}\tddeux{$\tiny{\sigmaup}$}{$w$})=&\tduns{$\tiny {\sigmaup}$}\tdun{$x$}\ot \tdquatretrois{$\tiny{\sigmaup}$}{$\tiny{\sigmaup}$}{$z$}{$y$}\tddeux{$\tiny{\sigmaup}$}{$w$}+\etree\ot \tddeux{$\tiny{\sigmaup}$}{$x$}\tdquatretrois{$\tiny{\sigmaup}$}{$\tiny{\sigmaup}$}{$z$}{$y$}\tddeux{$\tiny{\sigmaup}$}{$w$}+\tduns{$\tiny {\sigmaup}$}\ot \tduns{$\tiny {\sigmaup}$}\tdquatretrois{$\tiny{\sigmaup}$}{$\tiny{\sigmaup}$}{$z$}{$y$}\tddeux{$\tiny{\sigmaup}$}{$w$}+\tdtroisun{$\tiny{\sigmaup}$}{$x$}{$\tiny{\sigmaup}$}
\tdun{$y$}\tddeux{$\tiny{\sigmaup}$}{$z$}\ot\tddeux{$\tiny{\sigmaup}$}{$w$}+\tdtroisun{$\tiny{\sigmaup}$}{$x$}{$\tiny{\sigmaup}$}\ot\tdtroisdeux{$\tiny{\sigmaup}$}{$\tiny{\sigmaup}$}{$z$}\tddeux{$\tiny{\sigmaup}$}{$w$}\\ &+\tdtroisun{$\tiny{\sigmaup}$}{$x$}{$\tiny{\sigmaup}$}\tdun{$y$}\tdun{$z$}\ot\tduns{$\tiny {\sigmaup}$}\tddeux{$\tiny{\sigmaup}$}{$w$}+\tdtroisun{$\tiny{\sigmaup}$}{$x$}{$\tiny{\sigmaup}$}\tdun{$y$}\ot\tddeux{$\tiny{\sigmaup}$}{$\tiny{\sigmaup}$}\tddeux{$\tiny{\sigmaup}$}{$w$}
+\tdtroisun{$\tiny{\sigmaup}$}{$x$}{$\tiny{\sigmaup}$}\tdquatretrois{$\tiny{\sigmaup}$}{$\tiny{\sigmaup}$}{$z$}{$y$}\tdun{$w$}\ot \etree+\tdtroisun{$\tiny{\sigmaup}$}{$x$}{$\tiny{\sigmaup}$}\tdquatretrois{$\tiny{\sigmaup}$}{$\tiny{\sigmaup}$}{$z$}{$y$}\ot\tduns{$\tiny {\sigmaup}$}.
\end{align*}
\end{exam}

Next we prove that the combinatorial definition of $\col$ in Eq.~(\mref{eq:eqcombint}) is the same as the recursive definition.
First note that they agree on the initial step:
\begin{align}
\col(\etree) = 0\,\text{ and }\, \col(\bullet_x) = \etree\ot \etree \,\text{ for }\, x\in X.
\mlabel{eq:eqinitial}
\end{align}
So it suffices to show that $\col$ in Eq.~(\mref{eq:eqcombint}) satisfies Eqs.~(\mref{eq:dbp}) and~(\mref{eq:dele1}), which determinate the induction step of the recursive definition of $\col$.

\begin{lemma}
Let $\col$ be given in Eq.~(\mref{eq:eqcombint}) and $T_{1},\cdots,T_{n} \in \calt_\ell(\tx)$ with $n\geq 2$. Then
\begin{equation*}
\col(T_{1}\cdots T_{n})=T_{1}\cdot\col(T_{2}\cdots T_{n})+ \col(T_{1})\cdot (T_{2}\cdots T_{n}).
\end{equation*}
\mlabel{le:lelip}
\end{lemma}

\begin{proof}
We have
\begin{align*}
\col(T_{1}\cdots T_{n})=&\,\sum_{a\in V\left(T_{1}\cdots T_{n}\right)} B_{a}\ot R_{a}\\
=&\,\sum_{a\in V\left(T_{2}\cdots T_{n}\right)}B_{a}\ot R_{a}+\sum_{a\in V\left(T_{1}\right)}B_{a}\ot R_{a}\\
=&\,\sum_{a\in V\left(T_{2}\cdots T_{n}\right)} (T_{1} B'_{a}) \ot R_{a} + \sum_{a\in V\left(T_{1}\right)}  B_{a}\ot (R'_{a} T_{2}\cdots T_{n})\\
=&\,\sum_{a\in V\left(T_{2}\cdots T_{n}\right)}T_{1}\cdot (B'_{a}\ot R_{a}) + \sum_{a\in V\left(T_{1}\right)} (B_{a}\ot R'_{a}) \cdot \left(T_{2}\cdots T_{n}\right)\\
=&\,T_{1}\cdot \left(\sum_{a\in V\left(T_{2}\cdots T_{n}\right)}B'_{a}\ot R_{a}\right)+\left(\sum_{a\in V\left(T_{1}\right)}B_{a}\ot R'_{a}\right)\cdot \left(T_{2}\cdots T_{n}\right)\\
=&\,T_{1}\cdot\col(T_{2}\cdots T_{n})+\col(T_{1})\cdot \left(T_{2}\cdots T_{n}\right)\quad (\text{by Eq.~(\mref{eq:eqcombint})}),
\end{align*}
where $B'_a$ is the intersection of $B_a$ with $T_2\cdots T_n$, and $R'_a$ is the intersection of $R_a$ with $T_1$.
\end{proof}

\begin{lemma}
Let  $\col$ be given in Eq.~(\mref{eq:eqcombint}) and $F= B^+(\lbar{F}) \in \calt_\ell(\tx)$. Then
\begin{equation}
\col (F) =\lbar{F}\ot 1+(\id\ot B^+)\col(\lbar{F}).
\mlabel{eq:eqcomb}
\end{equation}
\mlabel{le:lecocycle}
\end{lemma}

\begin{proof}
If $\lbar{F}=1$, then $F = \bullet_\sigmaup$ and $\col(F) = \col(\bullet_\sigmaup)=1\ot 1$ by Eq.~(\mref{eq:eqcombint}).
Since $\col(\lbar{F})= \col(\etree) = 0$ by Eq.~(\mref{eq:eqinitial}), Eq.~(\mref{eq:eqcomb}) is valid for $\lbar{F}=1$.

Consider $\lbar{F}\neq 1$. Then $\lbar{F}=T_{1}\cdots T_{n}$ for some $T_{1}, \cdots,  T_{n}\in \calt_\ell(\tx)\setminus \{\etree\}$ with $n\geq 1$. We have
\begin{align*}
\col(F)=&\,\sum_{a\in V(F)}B_{a}\ot R_{a}=\,\lbar{F}\ot 1+ \sum_{a\in V (\lbar{F})}B_{a}\ot R_{a}\\
=&\,\lbar{F}\ot 1+ \sum_{i=1}^n \sum_{a\in V(T_{i})}B_{a}\ot R_{a}\\
=&\,\lbar{F}\ot 1+ \sum_{i=1}^n \sum_{a\in V(T_{i})} (T_1\cdots T_{i-1}B'_{a}) \ot B^+(R_a' T_{i+1}\cdots T_n) \\
=&\,\lbar{F}\ot 1+ (\id \ot B^+) \left( \sum_{i=1}^n \sum_{a\in V(T_{i})} (T_1\cdots T_{i-1}B'_{a}) \ot (R_a' T_{i+1}\cdots T_n)\right) \\
=&\,\lbar{F}\ot 1+ (\id \ot B^+) \left( \sum_{i=1}^n \sum_{a\in V(T_{i})} (T_1\cdots T_{i-1})\cdot (B'_{a} \ot R_a')\cdot  (T_{i+1}\cdots T_n)\right) \\
=&\,\lbar{F}\ot 1+ (\id \ot B^+) \left( \sum_{i=1}^n (T_1\cdots T_{i-1})\cdot \left(\sum_{a\in V(T_{i})}B'_{a} \ot R_a' \right)\cdot  (T_{i+1}\cdots T_n)\right) \\
=&\,\lbar{F}\ot 1+(\id\ot B^+)\left(\sum_{i = 1}^n \left(T_{1}\cdots T_{i-1}\right)\cdot\col(T_{i})\cdot\left( T_{i+1}\cdots T_{n}\right)\right)\quad\left(\text{by Eq.~(\mref{eq:eqcombint}})\right)\\
=&\,\lbar{F}\ot 1+(\id\ot B^+)\col(T_{1}\cdots T_{n})\quad(\text{by the repetition of Lemma~\mref{le:lelip}})\\
=&\,\lbar{F}\ot 1+(\id\ot B^+)\col(\lbar{F}),
\end{align*}
where $B'_a$ (resp. $R'_a$) is the intersection of $B_a$ (resp. $R_a$) with $T_i$, and the fourth step employs the fact
that $R_a$ has the root of $F$. This completes the proof.
\end{proof}

\begin{prop}
The combinatorial description of $\col$ given in Eq.~(\mref{eq:eqcombint}) coincides with the recursive definition of $\col$ in Subsection~\mref{subs:copro}.
\end{prop}

\begin{proof}
It follows from Eq.~(\mref{eq:eqinitial}) and Lemmas~\mref{le:lelip}, \mref{le:lecocycle}.
\end{proof}

\subsubsection{Infinitesimal unitary bialgebras}
In order to provide an algebraic framework for the calculus of divided differences,
Joni and Rota~\mcite{JR} introduced the concept of an infinitesimal bialgebra. We adapt it to the following unitary version.

\begin{defn}
An {\bf infinitesimal unitary bialgebra} (abbreviated $\epsilon$-unitary bialgebra) is a quadruple $(A,m,1,\Delta)$,
where $(A,m,1)$ is a unitary associative algebra, $(A,\Delta)$ is a coassociative coalgebra (without counit), and for each $a,b\in A$,
\begin{equation}
\Delta (ab)=a\cdot \Delta(b)+\Delta(a) \cdot b.
\mlabel{eq:cocycle}
\end{equation}
\end{defn}

Here $(A,\Delta)$ cannot have the counit,
since there is no non-zero infinitesimal bialgebra which is both unitary and counitary~\mcite{Ag0}.

\begin{defn}
Let $(A,m_{A},1_{A},\Delta_{A})$ and  $(B,m_{B},1_{B},\Delta_{B})$ be $\epsilon$-unitary bialgebras.
A map $\psi : A\rightarrow B$ is called an {\bf infinitesimal unitary bialgebra morphism} (abbreviated $\epsilon$-unitary bialgebra morphism) if
$\psi$ is a unitary algebra morphism and a coalgebra morphism.
\end{defn}

Now we arrive at our main result in this subsection.

\begin{theorem}
The quadruple $(\hlf(\tx),\, \conc, \etree, \,\col )$ is an  $\epsilon$-unitary bialgebra.
\mlabel{thm:rt2}
\end{theorem}

\begin{proof}
It is known that the triple $(\hlf(\tx), \,\mul,\etree)$ is a unitary algebra~\cite[Theorem~2.8]{ZGG}.
Furthermore, $(\hlf(\tx),\, \col )$ is a coalgebra by Theorem~\mref{thm:rt1}.
Finally Eq.~(\mref{eq:cocycle}) follows from Lemma~\mref{lem:colff}.
This completes the proof.
\end{proof}

\subsection{Free cocycle infinitesimal unitary bialgebras}
In this subsection, we construct the  free cocycle infinitesimal unitary bialgebra on a set. For this, let us pose the following concepts.
\begin{defn}
\begin{enumerate}
\item  An  {\bf operated  infinitesimal unitary bialgebra} (abbreviated operated $\epsilon$-unitary bialgebra) $(H,\,m,\,1, \Delta, P)$ is an $\epsilon$-unitary bialgebra  $(H,\,m,\,1, \Delta)$ which is also an operated algebra $(H,\,P)$.

\item Let $(A,\,P_A)$ and $(B,\,P_B)$ be two operated $\epsilon$-unitary bialgebras.
A map $\psi : A\rightarrow B$ is called an {\bf operated infinitesimal unitary bialgebra morphism} (abbreviated operated $\epsilon$-unitary bialgebra morphism)
if $\psi$ is an $\epsilon$-unitary bialgebra morphism and $\psi P_A = P_B \psi$.

\item A {\bf cocycle infinitesimal unitary bialgebra} (abbreviated cocycle $\epsilon$-unitary bialgebra) is an operated  $\epsilon$-unitary bialgebra
$(H,\,m,\,1, \Delta, P)$  satisfying the $\epsilon$-cocycle condition:
\begin{equation}
\Delta P = \id \otimes 1+ (\id\otimes P) \Delta.
\mlabel{eq:eqiterated}
\end{equation}

\item The {\bf free cocycle $\epsilon$-unitary bialgebra on a set $X$} is a cocycle $\epsilon$-unitary bialgebra
$(H_{X},\,m_{X},\\ \,1_X, \Delta_{X}, \,P_X)$ together with a set map $j_X:X\to H_{X}$ with the property that,
for any cocycle $\epsilon$-unitary bialgebra $(H,\,m,\,1, \Delta,\,P)$ and set map $f:X\to H$
such that $\Delta(f(x))=1\ot 1$ for $x\in X$, there is a unique morphism $\free{f}:H_X\to H$ of
operated $\epsilon$-unitary bialgebras such that $\free{f} j_X=f$.
\mlabel{it:def4}
\end{enumerate}
\mlabel{de:defciub}
\end{defn}

\begin{theorem} Let $j_{X}: X\hookrightarrow \hlf(\tx), ~x \mapsto \bullet_x$ be the natural embedding.
Then the quintuple $(\hlf(\tx), \,\mul,\,1, \col,\,B^+)$ together with $j_X$ is the free cocycle $\epsilon$-unitary bialgebra on the set $X$.
\mlabel{thm:propm}
\end{theorem}

\begin{proof}
The quadruple $(\hlf(\tx), \,\mul,\,1, \col)$ is an $\epsilon$-unitary bialgebra by Theorem~\mref{thm:rt2},
and further, equipped with the operator $B^+$, is a cocycle $\epsilon$-unitary bialgebra by Eq.~(\mref{eq:dbp}).

Let $(H,\, m,\,1_H, \Delta,\, P)$ be a cocycle $\epsilon$-unitary bialgebra and let $f:X\to H$ be a map such that $\Delta(f(x))=1\ot 1$ for $x\in X$.
Then $(H,m,1_H, P)$ is an operated unitary algebra.
Note that $(\hlf(\tx), \,\mul,\,1_H,\, B^+)$ is the free operated unitary algebra on $X$~\mcite{ZGG}.
So there is a unique operated unitary algebra morphism $\free{f}:\hlf(\tx) \to H$ such that $\free{f} j_X=f$.
It remains to check the compatibility of the coproducts $\Delta$ and $\col$ for which we verify
\begin{equation}
\Delta \free{f} (F)=(\free{f}\ot \free{f}) \col (F)\quad \text{for all } F\in \calf_\ell(\tx),
\mlabel{eq:copcomp}
\end{equation}
by induction on the depth $\dep(F)\geq 0$. If $\dep(F)=0$, we have $F = \bullet_{x_1} \cdots \bullet_{x_m}$ for some $m\geq 0$ and $x_1, \cdots, x_m\in X$.
Then
\allowdisplaybreaks{
\begin{align*}
\Delta\free{f}(\bullet_{x_1} \cdots \bullet_{x_m})=&\ \Delta \Big( \free{f}(\bullet_{x_1})\cdots\free{f}(\bullet_{x_{m}}) \Big)\\
=&\ \sum_{i=1}^m \Big(\free{f}(\bullet_{x_{1}})\cdots\free{f}(\bullet_{x_{i-1}})\Big)\cdot
\Delta\big(\free{f}(\bullet_{x_{i}})\big) \cdot \Big(\free{f}(\bullet_{x_{i+1}})\cdots\free{f}(\bullet_{x_{m}})\Big)
\quad \text{(by Eq.~(\ref{eq:cocycle}))}\\
=&\ \sum_{i=1}^m \Big(\free{f}(\bullet_{x_{1}})\cdots\free{f}(\bullet_{x_{i-1}})\Big)\cdot
\Delta\big(f(x_{i})\big) \cdot \Big(\free{f}(\bullet_{x_{i+1}})\cdots\free{f}(\bullet_{x_{m}})\Big)\\
=&\ \sum_{i=1}^m \Big(\free{f}(\bullet_{x_{1}})\cdots\free{f}(\bullet_{x_{i-1}})\Big)\cdot
(1\ot 1)\cdot \Big(\free{f}(\bullet_{x_{i+1}})\cdots\free{f}(\bullet_{x_{m}})\Big)\\
=&\ \sum_{i=1}^m\left(\free{f}(\bullet_{x_{1}})\cdots\free{f}(\bullet_{x_{i-1}})\right)\ot\left(\free{f}(\bullet_{x_{i+1}}) \cdots\free{f}(\bullet_{x_{m}})\right)\\
=&\ \sum_{i=1}^m\free{f}(\bullet_{x_{1}}\cdots\bullet_{x_{i-1}})\ot\free{f}(\bullet_{x_{i+1}} \cdots\bullet_{x_{m}})\\
=&\ \sum_{i=1}^m (\free{f}\ot\free{f})(\bullet_{x_{1}}\cdots \bullet_{x_{i-1}}\ot \bullet_{x_{i+1}} \cdots \bullet_{x_{m}})\\
=&\ (\free{f}\ot\free{f}) \left(\sum_{i=1}^m \bullet_{x_{1}}\cdots \bullet_{x_{i-1}}\ot \bullet_{x_{i+1}} \cdots \bullet_{x_{m}}\right)\\
=&\ (\free{f}\ot\free{f})\col(\bullet_{x_1} \cdots \bullet_{x_m})\quad(\text{by Lemma~\ref{lem:rt11}}).
\end{align*}
}

Assume that Eq.~(\mref{eq:copcomp}) holds for $\dep(F)\leq n$ and consider the case of $\dep(F)=n+1$. For this case we apply the induction on the breadth $\bre(F)\geq 1$. If $\bre(F)=1$, since $\dep(F)=n+1\geq1$, we have $F=B^+(\lbar{F})$ for some
$\lbar{F}\in\ldf(\tx)$.  Then
\allowdisplaybreaks{
\begin{align*}
\Delta \free{f}(F)&=\Delta \free{f} (B^{+}(\lbar{F}))=\Delta P(\free{f} (\lbar{F}))\\
&=\free{f}(\lbar{F})\ot 1_{H}+ (\id\ot P)\Delta(\free{f} (\lbar{F}))\quad(\text{by Eq.~(\mref{eq:eqiterated})})\\
&=\free{f}(\lbar{F})\ot 1_{H}+ (\id\ot P)(\free{f}\ot \free{f}) \col (\lbar{F}) \quad(\text{by the induction hypothesis on~}\dep(F)) \\
&=\free{f}(\lbar{F})\ot 1_{H}+ (\free{f}\ot P\free{f}) \col (\lbar{F})\\
&=\free{f}(\lbar{F})\ot 1_{H}+ (\free{f}\ot \free{f}B^+) \col (\lbar{F}) \quad(\text{by $\lbar{f}$ being an operated algebra morphism}) \\
&=(\free{f}\ot \free{f})\Big(\lbar{F}\ot 1+(\id\ot B^+)\col (\lbar{F})\Big) \\
&=(\free{f}\ot \free{f}) \col (B^+(\lbar{F}))\\
&=(\free{f}\ot \free{f}) \col (F).
\end{align*}
}
Assume that Eq.~(\mref{eq:copcomp}) holds for $\dep(F)=n+1$ and $\bre(F)\leq m$, and consider the case when $\dep(F)=n+1$ and $\bre(F)=m+1\geq 2$. Then $F=F_{1}F_{2}$ for some $F_{1},F_{2}\in\ldf(\tx)$ with $\bre(F_{1}), \bre(F_{2}) < \bre(F)$. Then
\allowdisplaybreaks{
\begin{align*}
\Delta \free{f}(F)=&\ \Delta \free{f} (F_{1}F_{2})=\Delta(\free{f}(F_1)\free{f}(F_2))\\
=&\ \free{f}(F_{1})\cdot \Delta (\free{f}(F_{2}))+\Delta(\free{f}(F_{1})) \cdot \free{f}(F_{2})\\
=&\ \free{f}(F_{1}) \cdot (\free{f}\ot \free{f})\col(F_{2})+(\free{f} \ot \free{f})\col(F_{1}) \cdot \free{f}(F_{2}) \quad (\text{by the induction on breadth})\\
=&\ \free{f}(F_{1}) \cdot (\free{f}\ot \free{f})\Big(\sum_{(F_{2})}F_{2(1)}\ot F_{2(2)}\Big)+(\free{f}\ot \free{f})\Big(\sum_{(F_{1})}F_{1(1)}\ot F_{1(2)}\Big) \cdot \free{f}(F_{2})\\
=&\ \free{f}(F_{1})\cdot \Big(\sum_{(F_{2})}\free{f}(F_{2(1)})\ot \free{f}(F_{2(2)})\Big)+\Big(\sum_{(F_{1})}\free{f}(F_{1(1)})\ot\free{f} (F_{1(2)})\Big) \cdot \free{f}(F_{2})\\
=&\ \sum_{(F_{2})}\free{f}(F_{1})\free{f}(F_{2(1)})\ot\free{f}(F_{2(2)})+\sum_{(F_{1})}\free{f}(F_{1(1)})\ot\free{f} (F_{1(2)})\free{f}(F_{2}) \quad (\text{by Eq.~(\ref{eq:dota})})\\
=&\ \sum_{(F_{2})}\free{f}(F_{1}F_{2(1)})\ot\free{f}(F_{2(2)})+\sum_{(F_{1})}\free{f}(F_{1(1)})\ot\free{f} (F_{1(2)}F_{2})\\
=&\ (\free{f}\ot \free{f})\Big(\sum_{(F_{2})}F_{1}F_{2(1)}\ot F_{2(2)}\Big)+(\free{f}\ot \free{f})\Big( \sum_{(F_{1})}F_{1(1)}\ot F_{1(2)}F_{2}\Big)\\
=&\ (\free{f}\ot \free{f})\left(\sum_{(F_{2})}F_{1}F_{2(1)}\ot F_{2(2)}+\sum_{(F_{1})}F_{1(1)}\ot F_{1(2)}F_{2} \right)\\
=&\ (\free{f}\ot \free{f})\left(F_{1}\cdot \sum_{(F_{2})}F_{2(1)}\ot F_{2(2)}+ \Big(\sum_{(F_{1})}F_{1(1)}\ot F_{1(2)}\Big) \cdot F_{2} \right)\\
=&\ (\free{f} \ot \free{f})(F_{1} \cdot \col(F_{2})+\col(F_{1})\cdot F_{2})\\
=&\ (\free{f}\ot \free{f})\col(F_{1}F_{2})  \quad(\text{by Lemma~\mref{lem:colff}})\\
=&\ (\free{f}\ot \free{f})\col(F).
\end{align*}
}
This completes the induction on the depth and hence the induction on the breadth.
\end{proof}

\section{Free cocycle infinitesimal unitary Hopf algebras of decorated planar rooted forests}\label{sec:hopf}
In the last section, we have proved that $\hlf(\tx)$ is the free cocycle $\epsilon$-unitary bialgebra on a set $X$.
In this section, we are going to show that it is further an $\epsilon$-unitary Hopf algebra and
then the free cocycle $\epsilon$-unitary Hopf algebra on a set $X$.
Throughout the remainder of the paper, we assume that $\bfk$ is a field with ${\rm char}(\bfk) =0$ and denote by $\Hom_{\bfk}(A,B)$ the set of linear map from $A$ to $B$.

The concept of an infinitesimal Hopf algebra was introduced by Aguiar in order to develop and study $\epsilon$-bialgebras~\mcite{Ag0}.
If $A$ is an $\epsilon$-bialgebra, then the space $\Hom_{\bfk}(A,A)$ is still an algebra under convolution:
 $$f\ast g=m(f\ot g)\Delta,$$
but it possibly without unit with respect to the convolution $\ast$~\mcite{Ag0}. So it is impossible to consider antipode.
To solve this difficulty, Aguiar equipped the space $\Hom_{\bfk}(A,A) $ with circular convolution $\circledast$ given by
$$f\circledast g :=f\ast g+f+g, \, \text {that is, } \, (f\circledast g)(a) :=\sum_{(a)}f(a_{(1)})g({a_{(2)}})+f(a)+g(a)\, \text{ for } a\in A.$$
Note that $f\circledast 0 = f = 0\circledast f$ and so $0\in \Hom_{\bfk}(A,A)$ is the unit with respect to the circular convolution $\circledast$.

Now we propose a unitary version of an infinitesimal Hopf algebra.

\begin{defn}
An infinitesimal unitary bialgebra $(A, m, 1, \Delta)$ is called an {\bf infinitesimal unitary Hopf algebra} (abbreviated {\bf $\epsilon$-unitary Hopf algebra}) if the identity map $\id\in \Hom_{\bfk}(A,A)$ is invertible with respect to the circular convolution. In this case, the inverse $S\in \Hom_{\bfk}(A,A)$ of $\id$ is called the {\bf antipode} of $A$. It is characterized by
the equations
\begin{equation}
\sum_{(a)}S(a_{(1)})a_{(2)}+S(a)+a=0=\sum_{(a)}a_{(1)}S(a_{(2)})+S(a)+a \, \text{ for }\, a \in A.
\mlabel{eq:eqha}
\end{equation}
where $\Delta(a) = \sum_{(a)} a_{(1)} \ot a_{(2)}$.
\mlabel{de:deha}
\end{defn}

\begin{remark}
Note $\Delta(1)=0$ by Eq.~(\mref{eq:cocycle}). In Eqs.~(\mref{eq:eqha}), we take $a=1$ to obtain
$S(1)=-1$.
\end{remark}

\begin{defn}
Let $(H, m_{H}, 1_{H},\Delta_{H})$ and $(L, m_{L}, 1_{L}, \Delta_{L})$ be $\epsilon$-unitary Hopf algebras,
and let $S_{H}$ and $S_{L}$ be the antipodes of $H$ and $L$, respectively. When an $\epsilon$-unitary bialgebra
morphism $\phi: H\rightarrow L$ satisfies the condition
$S_{L} \phi=\phi S_{H}$, $\phi$ is called an {\bf $\epsilon$-unitary Hopf algebra morphism}.
\end{defn}

\begin{exam}
Aguiar~\mcite{Ag0} verified that the polynomial algebra $\bfk[x]$ is an $\epsilon$-bialgebra satisfying
 \begin{align*}
\Delta(1)=0,\,\,\Delta(x^{n})=x^{n-1}\ot 1+x^{n-2}\ot x+\cdots+x\ot x^{n-2}+1\ot x^{n-1}\,\, \text{for}\,\,n\geq 1,
\end{align*}
and further an $\epsilon$-Hopf algebra with the antipode $S$ given by
\begin{align*}
S(x^{n})=-(x-1)^{n}\,\, \text{for}\,\,n\geq 0.
\end{align*}
Involved with the unity, it is also an $\epsilon$-unitary bialgebra and an $\epsilon$-unitary Hopf algebra.
\end{exam}

Based on Definitions~\mref{de:defciub} and~\mref{de:deha}, we pose the following concepts.

\begin{defn}
\begin{enumerate}
\item  An  {\bf operated  infinitesimal unitary Hopf algebra} (abbreviated operated $\epsilon$-unitary Hopf algebra) $(H,\,m,\,1, \Delta, P)$ is an $\epsilon$-unitary Hopf algebra $(H,\,m,\,1, \Delta)$ which is also an operated algebra $(H,\,P)$.

\item
Let $(H, m_{H}, 1_{H},\Delta_{H}, P_{H})$ and $(L, m_{L}, 1_{L}, \Delta_{L}, P_{L})$ be two operated $\epsilon$-unitary Hopf algebras. A map $\phi: H\rightarrow L$ is called an {\bf operated infinitesimal unitary Hopf algebra morphism} (abbreviated operated $\epsilon$-unitary Hopf algebra morphism) if $\phi$ is an
$\epsilon$-unitary Hopf algebra morphism and $\phi P_{H}=P_{L}\phi$.

\item
If the cocycle $\epsilon$-unitary bialgebra is an $\epsilon$-unitary Hopf algebra, then it is called a {\bf cocycle infinitesimal unitary Hopf algebra} (abbreviated cocyle $\epsilon$-unitary Hopf algebra).
\item
The {\bf free cocyle $\epsilon$-unitary Hopf algebra on a set $X$} is a cocycle $\epsilon$-unitary Hopf algebra $(H_{X},\,m_{X},\,1_X, \Delta_{X}, \,P_X)$ together with a set map $j_X:X\to H_{X}$ with the property that, for any cocycle $\epsilon$-unitary Hopf algebra $(H,\,m,\,1, \Delta,\,P)$ and set map $f:X\to H$ such that $\Delta(f(x))=1\ot 1$ for $x\in X$, there is a unique morphism $\free{f}:H_X\to H$ of operated $\epsilon$-unitary Hopf algebras such that $\free{f} j_X=f$.
\end{enumerate}
\mlabel{defn:ciuha}
\end{defn}

We expose some notations and results as preparation.

\begin{defn}\cite[p.~11]{Ag0}
Let $A$ be an algebra and $C$ a coalgebra. The map $f:C\to A$ is called {\bf locally nilpotent} with respect to convolution $*$
if for each $c\in C$ there is some $n\geq 1$ such that
\begin{equation*}
f^{\ast n}(c) :=\sum_{(c)}f(c_{(1)})f(c_{(2)})\cdots f(c_{(n+1)})=0.
\end{equation*}
\end{defn}

Denote by $\mathbb{R}$ and $\mathbb{C}$ the field of real numbers and the field of complex numbers, respectively.

\begin{lemma}\cite[Proposition~4.5]{Ag0}
Let $(A,\,m,\,\Delta)$ be an $\epsilon$-bialgebra over a field $\bfk$ and $D :=m\Delta$. Suppose that either
\begin{enumerate}
\item
$\bfk = \mathbb{R}$ or $\mathbb{C}$ and $A$ is finite dimensional, or
\item
$D$ is locally nilpotent and char$(\bfk)=0$.
\end{enumerate}
Then $A$ is an $\epsilon$-Hopf algebra with bijective antipode $S=-\sum_{n=0}^{\infty}\frac{1}{n!}(-D)^{n}$.
\mlabel{lem:rt3}
\end{lemma}

We proceed to prove that the map $m_{RT}\col$ on $\hlf(\tx)$ is locally nilpotent. For this, denote by
$$H^{n}:= H_\ell^{n}(\tx) := \bfk \left\{F\in \ldf(\tx) \medmid |F| = n \right\}\, \text{ for } n\geq 0,$$
where $|F|$ is the number of vertices of $F$.

\begin{lemma}
For each $F\in H^{n}$ with $n\geq 1$, we have $\col(F)\in \sum_{p+q=n-1}H^{p}\ot H^{q}.$
\mlabel{lem:rt122}
\end{lemma}

\begin{proof}
It follows from Eq.~(\mref{eq:eqcombint}) and the fact that $|B_a| + |R_a| = |F|-1$ for each $a\in V(F)$.
\end{proof}

\begin{lemma}
Let $(\hlf(\tx), \,\conc,\,\etree,\col)$ be the $\epsilon$-unitary bialgebra as in Theorem~\mref{thm:rt2}
and $$D_{\epsilon}:= \conc \col: \hlf(\tx) \to \hlf(\tx).$$
Then for each $n\geq 0$ and $F\in H^{n}$, $D_{\epsilon}^{\ast (n+1)}(F)=0$ and so $D_{\epsilon}$ is locally nilpotent.
\mlabel{lem:rt12}
\end{lemma}

\begin{proof}
We prove the result by induction on $n\geq 0$. For the initial step of $n=0$, it follows from Eq.~(\mref{eq:dele}) that
\begin{equation}
D_{\epsilon}(F)=D_{\epsilon}(1)=\mul\col(1)=0.
\mlabel{eq:eqdu1}
\end{equation}

Assume that $D_{\epsilon}^{\ast (n+1)}(F)=0$ holds for $F\in H^{n}$ with $n<k$, and consider the case when $n=k.$
Suppose first that $\bre(F) =1$. If $F=\bullet_x$ for some $x\in X$, then $F\in H^1$ and so
\begin{align*}
D_{\epsilon}^{\ast 2}(F)=\mul(D_{\epsilon}\ot D_{\epsilon})\col(\bullet_{x})=\mul(D_{\epsilon}\ot D_{\epsilon})(1\ot 1)=D_{\epsilon}(1)D_{\epsilon}(1)
=0.
\end{align*}
If $F\neq \bullet_x$ for all $x\in X$, then we can write $F=B^+(\lbar{F})$ for some $\lbar{F}\in H^{k-1}$.
Thus
\begin{align*}
D_{\epsilon}^{\ast (k+1)}(F)=&\ (D_{\epsilon}^{\ast k}\ast D_{\epsilon})(F)=\mul (D_{\epsilon}^{\ast k}\ot D_{\epsilon})\col(F)= \mul (D_{\epsilon}^{\ast k}\ot D_{\epsilon})\col\Big(B^+(\lbar{F})\Big)\\
=&\ \mul(D_{\epsilon}^{\ast k}\ot D_{\epsilon})\left(\lbar{F}\ot 1+(\id\ot B^+)\col(\lbar{F})\right) \quad(\text{by Eq.~(\mref{eq:dbp}})) \\
=&\ \mul\left(D_{\epsilon}^{\ast k}(\lbar{F})\ot D_{\epsilon}(1)\right)+\mul (D_{\epsilon}^{\ast k}\ot D_{\epsilon}B^+)\col(\lbar{F})\\
=&\ \mul (D_{\epsilon}^{\ast k}\ot D_{\epsilon}B^+)\col(\lbar{F})\quad(\text{by Eq.~(\mref{eq:eqdu1}}))\\
=&\ \mul (D_{\epsilon}^{\ast k}\ot D_{\epsilon}B^+)\Big(\sum_{(\lbar{F})}\lbar{F}_{(1)}\ot\lbar{F}_{(2)}\Big)\\
=&\ \mul\Big(\sum_{(\lbar{F})}D_{\epsilon}^{\ast k}(\lbar{F}_{(1)})\ot D_{\epsilon}B^+(\lbar{F}_{(2)})\Big)\\
=&\ 0,
\end{align*}
where the last step employs the induction hypothesis and the fact that $|\lbar{F}_{(1)}|< |\lbar{F}| = k-1$.
Suppose next that $\bre(F) \geq 2$. Then we may write $F=F_{1}F_{2}$ with $\bre(F_1), \bre(F_2) < \bre(F)$.
Hence
\begin{align*}
D_{\epsilon}^{\ast (k+1)}(F)=&\ (D_{\epsilon}^{\ast k}\ast D_{\epsilon})(F)= (D_{\epsilon}^{\ast k}\ast D_{\epsilon})(F_{1}F_{2})= \mul (D_{\epsilon}^{\ast k}\ot D_{\epsilon})\col(F_{1}F_{2}) \\
=&\ \mul (D_{\epsilon}^{\ast k}\ot D_{\epsilon})(F_{1}\cdot\col(F_{2})+\col(F_{1})\cdot F_{2})\quad(\text{by Lemma~\mref{lem:colff}})\\
=&\ \mul (D_{\epsilon}^{\ast k}\ot D_{\epsilon})\Big(\sum_{(F_{2})}F_{1}F_{2(1)}\ot F_{2(2)}+\sum_{(F_{1})}F_{1(1)}\ot F_{1(2)}F_{2} \Big)\\
=&\ \mul\Big(\sum_{(F_{2})}D_{\epsilon}^{\ast k}(F_{1}F_{2(1)})\ot D_{\epsilon}(F_{2(2)})+\sum_{(F_{1})}D_{\epsilon}^{\ast k}(F_{1(1)})\ot D_{\epsilon}(F_{1(2)}F_{2}) \Big).
\end{align*}
By Lemma~\mref{lem:rt122},
\begin{equation*}
|F_{1}F_{2(1)}|=|F_{1}|+|F_{2(1)}|<|F_{1}|+|F_{2}|=|F| =k,
\end{equation*}
whence $D_{\epsilon}^{\ast k}(F_{1}F_{2(1)})=0$ by the induction hypothesis.
Similarly,
$$|F_{1(1)}|<|F_1|<|F_1| + |F_2| = |F| = k$$ and so $D^{\ast k}_{\epsilon}(F_{1(1)})=0$ by the induction hypothesis. Hence
$ D_{\epsilon}^{\ast (k+1)}(F)=0.$
This completes the proof.
\end{proof}

The following result shows that $\hlf(\tx)$ has an $\epsilon$-unitary Hopf algebraic structure.

\begin{theorem}
The quadruple $(\hlf(\tx), \,\conc,\,\etree, \col)$  is an $\epsilon$-unitary Hopf algebra with bijective antipode $S=-\sum_{n=0}^{\infty}\frac{1}{n!}(-D_{\epsilon})^{n}$.
\mlabel{thm:rt13}
\end{theorem}

\begin{proof}
By Theorem~\mref{thm:rt2}, $(\hlf(\tx), \,\conc,\,\etree, \col)$ is an $\epsilon$-unitary bialgebra.
From Lemmas~\mref{lem:rt3}, \mref{lem:rt12} and our assumption that $\bfk$ being a field with ${\rm char}(\bfk) = 0$,
$(\hlf(\tx), \,\conc,\,\col)$ is an $\epsilon$-Hopf algebra with bijective antipode $S=-\sum_{n=0}^{\infty}\frac{1}{n!}(-D_{\epsilon})^{n}$.
So the result holds by Definition~\mref{de:deha}.
\end{proof}

The following lemma is needed.

\begin{lemma}\cite[Proposition~3.8]{Ag0}
Let $H$ and $L$ be $\epsilon$-Hopf algebras and $\phi: H\rightarrow L$ a morphism of $\epsilon$-bialgebras. Then
$\phi S_{H}=S_{L}\phi$, i.e. $\phi$ is a morphism of $\epsilon$-Hopf algebras.
\mlabel{lem:rt16}
\end{lemma}

Now, we arrive at our main result of this subsection.

\begin{theorem}
Let $j_{X}: X\hookrightarrow \hlf(\tx), ~x \mapsto \bullet_x$ be the natural embedding. Then
the quintuple $(\hlf(\tx), \,\mul,\,\etree,\, \col,\,B^+)$ together with the $j_X$ is the free cocycle $\epsilon$-unitary Hopf algebra on the set $X$.
\mlabel{thm:rt16}
\end{theorem}

\begin{proof}
The $(\hlf(\tx), \,\mul,\,\etree,\, \col)$ is an $\epsilon$-unitary Hopf algebra by Theorem~\mref{thm:rt13},
and further, together with the operator $B^+$, is a cocycle $\epsilon$-unitary Hopf algebra by Eq.~(\mref{eq:dbp}).

Let $(H,m,1_H,\Delta,P)$ be a cocycle $\epsilon$-unitary Hopf algebra, where the antipode is suppressed, and let $f: X\rightarrow H$ be a set map
such that $\Delta(f(x)) = 1_H\ot 1_H$ for $x\in X$.
By Theorem ~\mref{thm:propm}, there is a unique morphism $\free{f}:\hlf(\tx)\rightarrow H$ of operated $\epsilon$-unitary bialgebras.
In particular, $\free{f}$ is a morphism of $\epsilon$-bialgebras. By Lemma~\mref{lem:rt16}, $\free{f}$ is compatible with the antipodes and so
is a morphism of operated $\epsilon$-unitary Hopf algebras. This proves the desired universal property.
\end{proof}

Let $X = \emptyset$ be the empty set. Then $\tx = X \sqcup\{\sigmaup\} = \{\sigmaup\}$ is a singleton set.
In this case, decorated planar rooted forests $\ldf(\tx)$ have the same decoration $\sigmaup$.
Equivalently, forests in $\ldf(\tx)$ have no decorations
and $\ldf(\tx)$ is precisely $\calf$. So we obtain a cocycle $\epsilon$-unitary Hopf algebraic structure on planar rooted forests,
which are the object studied in the Foissy-Holtkamp Hopf algebra~\mcite{Fo1, Hol}.

\begin{coro}
The quintuple $(\bfk \calf,\,\mul,\,1,\col,\,B^{+})$ is the free cocycle  $\epsilon$-unitary Hopf algebra on the empty set,
that is, the initial object in the category of cocycle $\epsilon$-unitary Hopf algebras.\mlabel{coro:wdeh}
\mlabel{coro:wdeh}
\end{coro}

\begin{proof}
It follows from Theorem~\mref{thm:rt16} by taking $X=\emptyset$.
\end{proof}

\bigskip

\noindent {\bf Acknowledgments}: This work was supported by the National Natural Science Foundation of
China (No.~11771191), the Fundamental Research Funds for the Central
Universities (No.~lzujbky-2017-162) and the Natural Science Foundation of Gansu Province (No.~17JR5RA175).

We thank the anonymous referee for valuable suggestions helping to improve the paper.

\end{document}